\documentclass[11pt]{article}

\RequirePackage{fullpage}
\RequirePackage{amsmath, amsfonts, amssymb, amsthm, amsbsy, amscd, bm, bbm,mathrsfs} 
\RequirePackage{graphicx}
\RequirePackage[small,bf]{caption}
\RequirePackage{subcaption}
\RequirePackage{hyperref, url, color}
\RequirePackage[normalem]{ulem}
\RequirePackage{algorithm, algorithmic}
\RequirePackage{relsize}
\RequirePackage{filecontents}
\usepackage{titling}
\usepackage{enumitem}
\usepackage{mathtools}

\allowdisplaybreaks

\def\iid{\textrm{i.i.d.~}} 
\DeclareMathOperator*{\argmin}{argmin}

\DeclareGraphicsExtensions{.eps,.pdf}
\graphicspath{{figs/}}

\newcommand{\R}{\mathbb{R}}
\newcommand{\Var}{\operatorname{\textnormal{Var}}}
\newcommand{\Cov}{\operatorname{\textnormal{Cov}}}

\newcommand{\goto}{\rightarrow}
\newcommand{\sgn}{\mathop\mathrm{sgn}}

\newcommand{\F}{\mathcal{F}}
\renewcommand{\P}{\operatorname{\mathbb{P}}}
\newcommand{\E}{\operatorname{\mathbb{E}}}

\renewcommand\d{\operatorname{\mathrm{d}}}

\newcommand{\bX}{\bm{X}}

\numberwithin{equation}{section}

\newtheorem{theorem}{Theorem}
\newtheorem{othertheorem}{othertheorem}[section]
\newtheorem{lemma}[othertheorem]{Lemma}
\newtheorem{corollary}[othertheorem]{Corollary}
\newtheorem{proposition}[othertheorem]{Proposition}

\theoremstyle{definition}
\newtheorem{definition}[othertheorem]{Definition}
\theoremstyle{remark}
\newtheorem{remark}[othertheorem]{Remark}
\theoremstyle{assumption}

\theoremstyle{definition}

\usepackage[table,xcdraw]{xcolor}
\newcommand{\ko}{\widetilde{\bm{X}}}

\newcommand{\that}{\hat{t}}

\newcommand{\calH}{\mathcal{H}}
\newcommand{\overbar}[1]{\mkern 1.5mu\overline{\mkern-1.5mu#1\mkern-1.5mu}\mkern 1.5mu}
\newcommand{\tpp}{\textnormal{TPP}}
\newcommand{\tppinfty}{\textnormal{tpp}}
\newcommand{\alphacv}{\alpha_{\textnormal{\scriptsize cv}}}
\newcommand{\taucv}{\tau_{\textnormal{\scriptsize cv}}}
\newcommand{\fdp}{\textnormal{FDP}}
\newcommand{\RR}{\mathbb{R}}
\newcommand{\EE}{\mathbb{E}}
\newcommand{\PP}{\mathbb{P}}
\newcommand{\fdpinfty}{\textnormal{fdp}}
\newcommand{\fdr}{\textnormal{FDR}}

\newcommand{\fdphat}{\widehat{\textnormal{FDP}}}
\newcommand{\lambdacv}{\hat{\lambda}_{\textnormal{\scriptsize cv}}}
\newcommand{\lambdacvinfty}{\lambda_{\textnormal{\scriptsize cv}}}
\newcommand{\fsp}{\textnormal{FSP}}
\newcommand{\tsp}{\textnormal{TSP}}

\renewcommand{\vec}[1]{\mathrm{#1}}

\title{A Power Analysis for Model-X Knockoffs with $\ell_{p}$-Regularized Statistics}

\thanksmarkseries{alph}
\author{Asaf Weinstein \thanks{School of Computer Science and Engineering, Hebrew University of Jerusalem, Israel} \and  Weijie J.~Su \thanks{Department of Statistics and Data Science, University of Pennsylvania} \and Ma{\l}gorzata Bogdan \thanks{Department of Mathematics, University of Wroclaw, Poland; Department of Statistics, Lund University, Sweden}\and Rina F.~Barber \thanks{Department of Statistics, University of Chicago} \and Emmanuel J.~Cand\`es \thanks{Department of Statistics and Department of Mathematics, Stanford University}}

\date{}

\begin{document}
\maketitle

\vspace{-.75cm}

\begin{abstract}
Variable selection properties of procedures utilizing penalized-likelihood estimates is a central topic in the study of high dimensional linear regression problems. Existing literature emphasizes the quality of ranking of the variables by such procedures as reflected in the receiver operating characteristic curve or in prediction performance. Specifically, recent works have harnessed modern theory of approximate message-passing (AMP) to obtain, in a particular setting, exact asymptotic predictions of the type I–type II error tradeoff for selection procedures that rely on $\ell_{p}$-regularized estimators. 

\smallskip
In practice, effective ranking by itself is often not sufficient because some calibration for Type I error is required. In this work we study theoretically the power of selection procedures that similarly rank the features by the size of an $\ell_{p}$-regularized estimator, but further use Model-X knockoffs to control the false discovery rate in the realistic situation where no prior information about the signal is available. In analyzing the power of the resulting procedure, we extend existing results in AMP theory to handle the pairing between original variables and their knockoffs. This is used to derive exact asymptotic predictions for power. We apply the general results to compare the power of  the knockoffs versions of Lasso and thresholded-Lasso selection, and demonstrate that in the i.i.d.~covariate setting under consideration, tuning by cross-validation on the augmented design matrix is nearly optimal. 
We further demonstrate how the techniques allow to analyze also the Type S error, and a corresponding notion of power, when selections are supplemented with a decision on the sign of the coefficient. 
\end{abstract}

\section{Introduction}\label{sec:intro}
Suppose that we observe a matrix $\bX\in \R^{n\times p}$ of the measurements of $p$ predictor variables on each of $n$ subjects, and a response vector $Y\in \R^n$, and assume that
\begin{equation}\label{eq:model-linear}
Y = \bX\beta + \xi,\ \ \ \ \ \ \ \ \ \ \xi \sim \mathcal{N}_n(0,\sigma^2 \boldsymbol{I}),
\end{equation}
where $\beta = (\beta_1,...,\beta_p)^\top$ and $\sigma^2$ are unknown. 
In many modern applications where the linear model is appropriate, $p$ is large and we may have a reason to believe a priori that $\beta_j$ is small in magnitude for most $j= 1,...,p$. 
For example, in genetics $X_{ij}$ might encode the state (presence or absence) of a specific genetic variant $j$ for individual $i$, and $Y_i$ measures a quantitative trait of interest. 
Typical cases entail the number $p$ of genetic variants in the millions but, for all we know about this kind of problems, only a small number of them may have significant explanatory power. 
Finding mutations which are in that sense important among the $p$ candidates, is key to investigating the causal mechanism regulating the trait. 

\smallskip
Following recent literature \cite[for example]{barber2015controlling,candes2018panning}, here we 
%
treat the problem formally as a multiple hypothesis testing problem with respect to the model \eqref{eq:model-linear}, where the null hypotheses to be tested are
$$
H_{0j}: \ \beta_j = 0,\ \ \ \ \ \ \ j\in \mathcal{H}\equiv \{1,...,p\}.
$$
Denote by $\mathcal{H}_0 \equiv \{j: \beta_j = 0\}$ the (unknown) subset of nulls, and denote by $\mathcal{S} \equiv \mathcal{H}\setminus \mathcal{H}_0$ the subset of nonnulls. 
In general, a multiple testing procedure uses the data to output an estimate $\widehat{\mathcal{S}}\subseteq \{1,...,p\}$ of $\mathcal{S}$. 
For any such procedure we define the {\it false discovery proportion} and the {\it true positive proportion} as
\begin{equation*}
\fdp \equiv \frac{|\widehat{\mathcal{S}}\cap
  \mathcal{H}_0|}{|\widehat{\mathcal{S}}|} \ \ \text{ and } \ \ \tpp \equiv \frac{|\widehat{\mathcal{S}}\cap \mathcal{S}|}{|\mathcal{S}|},
\end{equation*}
respectively, with the convention $0/0\equiv 0$. 
A good testing procedure is one for which TPP is large and FDP is small, meaning that the test is able to separate nonnulls from nulls. 
We will later be concerned with the concrete problem of controlling the {\it false discovery rate}, 
\begin{equation*}
\fdr \equiv \E[\fdp],
\end{equation*}
below a prespecified level, and we say that a test is {\it valid at level $q$} if $\fdr\leq q$ for all $\beta$. 
Note that, per definition, any variable selection procedure qualifies as a testing procedure and vice versa, and we will use the two terms interchangeably. 

\subsection{Selecting variables by thresholding regularized estimators}\label{subsec:bridge}
With a growing interest in high-dimensional (large $p$) settings, considerable attention has been given over the past two decades to variable selection procedures relying on the Lasso program, 
\begin{equation}\label{eq:lasso}
\underset{b\in \mathbb{R}^p}{\text{minimize}}\ \frac{1}{2}\|Y-\bX \vec{b}\|_2^2 + \lambda\|\vec{b}\|_1. 
\end{equation}
The Lasso is appealing because it is relatively easy to solve and at the same time the solution to \eqref{eq:lasso} tends to be sparse. 
Thus, for any $\lambda>0$, if $\widehat{\beta}(\lambda)$ denotes the solution to \eqref{eq:lasso}, variable selection is readily elicited by associating with $\widehat{\beta}(\lambda)$ the subset
\begin{equation}\label{eq:lasso-var-selection}
\widehat{\mathcal{S}}\equiv \{j:\widehat{\beta}_j(\lambda)\neq 0\}, 
\end{equation}
which will be referred to as {\it Lasso selection} for the rest of this paper. 
Many works have studied the properties of Lasso selection, mostly establishing conditions on $\bX$ and $\beta$ for selection consistency, $\mathbb{P}(\widehat{\mathcal{S}}=\mathcal{S}) \to 1$, e.g. \cite{meinshausenhigh, zhao,zou,  wainwright2009, buhlmannlivre, dossal}. 
Such conditions turn out to be generally very stringent even in the noiseless case, $\sigma^2=0$; in other words, the fundamental phenomenon is not a matter of insufficient signal-to-noise ratio. 
While the conditions for \eqref{eq:lasso-var-selection} to recover a {\it superset} of the true support $\mathcal{S}$, also referred to as {\it screening}, are considerably less restrictive, it tends to select too many null variables (see, e.g., \cite{yezhang:10, buhlmannlivre, zhou2009thresholding}). 

\smallskip
This rather discouraging fact has motivated practitioners and theoreticians alike to consider as an alternative the procedure that takes into account the {\it magnitude} of the estimate by setting
\begin{equation}\label{eq:thresh-lasso}
\widehat{\mathcal{S}} \equiv \{j: |\widehat{\beta}_j(\lambda)| > t\}, 
\end{equation}
for some threshold $t>0$ \cite{zhou2009thresholding, yezhang:10, van2011adaptive, geer:11, TardivelBogdan2018}, to which we refer from now on as {\it thresholded-Lasso selection}. 
Even more generally, one may consider, as in \cite{wang2017bridge}, replacing the Lasso estimator $\widehat{\beta}(\lambda)$ in \eqref{eq:thresh-lasso} with some {\it bridge} estimator, 
\begin{equation}\label{eq:beta-bridge}
\widehat{\beta}(\gamma, \lambda) = \underset{b\in \mathbb{R}^p}{\text{minimize}}\ \frac{1}{2}\|Y-\bX \vec{b}\|_2^2 + \lambda\|\vec{b}\|_\gamma^\gamma, \ \ \ \ \ \gamma\geq 1,
\end{equation}
where $\|\vec{b}\|_\gamma^\gamma\equiv \sum|\beta_j|^{\gamma}$, and we use the symbol $\gamma$ from here on instead of the more standard notation $p$ (as in the title) because $p$ is already taken (denotes the number of columns in $\bX$). 
The optimization problem \eqref{eq:beta-bridge} retains computational convenience because it is still convex, and at the same time produces a richer family of {\it thresholded-bridge} selection procedures,
\begin{equation}\label{eq:thresh-bridge}
\widehat{\mathcal{S}} \equiv \{j: |\widehat{\beta}_j(\gamma,\lambda)| > t\},
\end{equation}
for some $\gamma\geq 1$ and a threshold $t>0$. 
In \cite{wang2017bridge} the above selection procedure is referred to as a {\it two-stage} variable selection technique, separating the {\it ranking} by the absolute value of the regularized regression estimator, and the {\it thresholding} at $t>0$. 

\medskip
In principle, the parametric curve that associates the expectations of FDP and TPP with every $\lambda>0$ for \eqref{eq:lasso-var-selection}, and with every $t> 0$ for \eqref{eq:thresh-bridge}, could be used to measure the quality of ranking for Lasso selection or for thresholded-bridge selection with different choices of $\gamma>1$. 
For fixed $n$ and $p$, however, there are no tractable forms for $\widehat{\beta}_j(\gamma,\lambda)$ in general (the special case $\gamma=2$ is an exception), and the expected FDP and TPP are also intractable functions of $\beta$ and $\sigma^2$. 

\smallskip
Remarkably, in a certain asymptotic regime and under some further modelling assumptions, it is possible to calculate the {\it limits} of FDP and TPP for \eqref{eq:lasso-var-selection} at any fixed $\lambda>0$, and for \eqref{eq:thresh-bridge} at any fixed $\gamma\geq 1$ and $t> 0$. 
More specifically, in a special case where $\bX$ has i.i.d.\ Gaussian entries, $p$ and $n$ grow comparably, and the sparsity is {\it linear}, $|\mathcal{S}| \approx \epsilon p$, 
\cite{SLOPE1} leveraged major advances from \cite{bayati2011dynamics, bayati2012} to first obtain {\it exact asymptotic predictions} of FDP and TPP for Lasso selection. 
In \cite{su2017false} a fundamental quantitative tradeoff between FDP and TPP for Lasso, valid uniformly in $\lambda$, was presented by extending the aforementioned results. 
Recently, \cite{wang2017bridge} obtained predictions of FDP and TPP for thresholded-bridge selection with any $\gamma\geq 1$, which covers in particular thresholded-Lasso selection. 


\smallskip
The main purpose in \cite{wang2017bridge} is to analyze the power corresponding to different choices of $\gamma\geq 1$ in \eqref{eq:thresh-bridge}, and compare them in different regimes of the signal. 
In particular, while the results of \cite{su2017false} imply that Lasso cannot achieve exact support recovery in this asymptotic setting, \cite{wang2017bridge} show that using thresholded-Lasso can improve dramatically the separation between null and nonnulls if $\lambda$ is chosen appropriately. 
This provides rigorous confirmation for the advantages of thresholded-Lasso, which have long been noticed by practitioners. 
Also, the analysis in \cite{wang2017bridge} reinforces the results of \cite{TardivelBogdan2018}, which imply that in the same asymptotic setting, 
thresholded-Lasso indeed achieves exact support recovery if the signal-to-noise ratio is high and the limiting signal sparsity is below the transition curve of \cite{DonTan05}. 

\subsection{A ``vertical" look at the Lasso path}\label{subsec:vertical}
Before proceeding to describe the main focus of this paper, we take a moment to reflect on the basic differences between Lasso selection and thresholded-Lasso selection that account for the potential power increase reported in, e.g., \cite{wang2017bridge}. 
At first glance, the two selection rules might not appear that different, because \eqref{eq:lasso-var-selection} is just \eqref{eq:thresh-lasso} with $t=0$. 
There is, however, a fundamental difference between Lasso and thresholded-Lasso. 
To illustrate this, we simulated data from the model with $n=100$, $p=200$, $\sigma=1$, and the coefficients are all zero except for $\beta_1=\ldots=\beta_{20}=10$.
Figure \ref{fig:paths} tracks the absolute value of the Lasso estimates $\widehat{\beta}_j(\lambda)$ as a function of $\lambda$, for null coefficients and for nonnull coefficients. 
In Lasso selection variables are collected in the order they become active, $\widehat{\beta}_j(\lambda)\neq 0$, as $\lambda$ decreases; pictorially, this corresponds to looking at the selection path ``horizontally" along the $\lambda$ axis. 
We can see that false discoveries occur early on the Lasso path (as studied and confirmed in \cite{su2017false}). 
Consequently, \eqref{eq:lasso-var-selection} cannot keep FDP small unless $\lambda$ is chosen large, which inevitably affects the power: in this example the maximum TPP for \eqref{eq:lasso-var-selection} subject to FDP$\leq 0.1$ is 0.45. 

\smallskip
Nevertheless, it is also evident from the figure that the estimates corresponding to true signals maintain significantly larger size than most of the estimates for nulls, as $\lambda$ decreases. 
This suggests that better separation between null and nonnulls can be achieved by looking further down the path (smaller $\lambda$) and ordering the variables according to the {\it magnitude} of the corresponding estimates; pictorially, this corresponds to looking at the selection path ``vertically", as represented by the dashed line at $\lambda=1.05$. 
The additional flexibility in varying the threshold allows \eqref{eq:thresh-lasso} to take advantage of this: 
basically, $\lambda$ can be chosen freely, while setting $t$ appropriately large will ensure small FDP (by killing small estimates corresponding to null coefficients). 
The potential advantage is demonstrated in the figure by the broken line, indicating the 10-fold cross-validation estimate of $\lambda$. 
At this value of $\lambda$, for example, thresholded-Lasso has TPP equal to 0.95 when $t$ is selected such that FDP $\leq 0.1$. 

\begin{figure}
  \centering
\includegraphics[scale=.6]{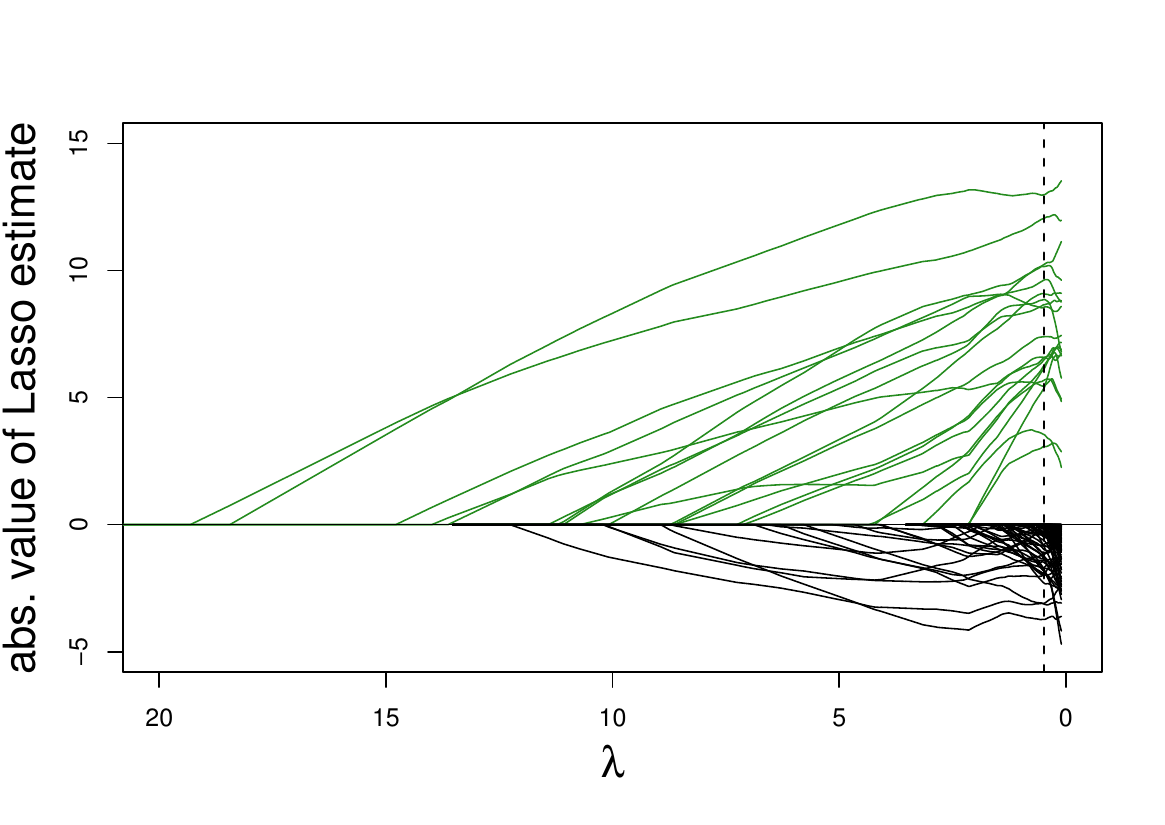}
\caption{
Lasso path for a simulated example. 
For convenience of presentation, the y-axis shows the absolute value of the lasso coefficient estimate for nonnull variables, and the negative of the absolute value for null variables. 
Vertical dashed line corresponds to $\lambda=1.05$.}
\label{fig:paths}
\end{figure}

\subsection{Calibration for Type I error}\label{subsec:controlled}

The works of \cite{su2017false} and \cite{wang2017bridge} are important because they facilitate a sharp theoretical comparison between Lasso selection and the thresholding  selection procedures \eqref{eq:thresh-bridge}. 
In practice, however, the implications are limited: the analysis in these works will yield the achievable asymptotic FDP for a prescribed asymptotic TPP level at any given $\lambda$ for \eqref{eq:lasso-var-selection}, and at any given $t$ for \eqref{eq:thresh-lasso}, {\it provided that $\sigma$ and the empirical distribution of the true coefficients $\beta_j$ are  known}. 
In reality, such {\it a priori} knowledge about the signal and the noise level is rarely available, and the FDP needs to be estimated instead. 
This motivated \cite{weinstein2017power} to study a knockoffs-augmented setup and obtain an operable counterpart to the ``oracle" FDP-TPP curve of \cite{su2017false} for Lasso selection. 
By ``operable" we mean that the power predictions of \cite{weinstein2017power} apply to a procedure that provably controls the FDR for fixed $n,p$ without any knowledge about $\beta$ or $\sigma$. 

\smallskip
Seeking to increase power while maintaining type I error control, in the present article we obtain an operable analog to the FDP and TPP predictions of \cite{wang2017bridge} for the thresholding selection procedures \eqref{eq:thresh-bridge}, with special attention given to {\it thresholded-Lasso} selection. 
As in \cite{weinstein2017power}, we employ knockoffs to allow for FDR calibration, observing that the augmented setup can still be studied within the same AMP framework. 
However, there is a crucial point of departure between our work and \cite{weinstein2017power} also in the {\it type} of knockoffs used: while the construction of \cite{weinstein2017power}, reviewed briefly in Section \ref{subsubsec:counting} and referred to as ``counting" knockoffs in the sequel, is valid only when the entries of $\bX$ are i.i.d., here we use the more general prescription of Model-X knockoffs 
from \cite{candes2018panning}. 
The counting knockoffs scheme studied in \cite{weinstein2017power} is something that the analyst would only implement if it were known that the covariates were i.i.d., as it controls FDR only in this limited setting. 
In contrast, the model-X knockoffs procedure is something that is widely used across a broad range of regimes, and has valid FDR control far beyond the i.i.d.~design setting. 
While our power analysis for this method is, at present, restricted to the i.i.d.~setting, the results in the current paper are far more useful since the analysis accommodates a much more general and broadly used kncokoff scheme (and the power analysis can hopefully be extended beyond the i.i.d.~setting in future work). 

\smallskip
To further justify studying Model-X knockoffs for i.i.d.~covariates, it is important to emphasize that---perhaps not obviously so---the i.i.d.~setting is very different from the orthogonal setting: for example, the discussions in \cite[Section 3.2.1]{SLOPE} and in \cite[Section 3]{su2017false} regarding Lasso, imply that due to shrinkage, even small sample correlations between the realized columns of $\bX$ generate additional ``noise'' as an artifact, which increases the variance of the estimates. 
This makes the analysis quite different, and more involved, as compared to the orthogonal $\bX$ case. 
Specifically, the level of this noise increases with the ratio $p/n$ and depends non-monotonically on the tuning parameter $\lambda$, see also Figures \ref{fig:simulation} and \ref{fig:cv} below. 
That explains, informally, why choosing an appropriate value of $\lambda$ (i.e., a value that makes the power large) is far from trivial already in the i.i.d.~covariate case under consideration here. 

\smallskip
Regarding the comparison with \cite{weinstein2017power}, besides the main difference mentioned above in the type of knockoffs, it is worth emphasizing that by implementing Model-X knockoffs, we also obviate the problem of estimating the proportion of nonnulls (i.e., the sparsity), which was a nontrivial issue to handle with counting knockoffs and involved an extra tuning parameter. 
Table \ref{tab:context} indicates where our work fits in the context of existing literature analyzing variable selection with bridge-penalized statistics in the AMP framework.

\subsection{Our contribution}\label{subsec:contribution}

As implied in the previous subsection, the thrust of this article is to develop mathematical tools enabling exact power analysis of Model-X knockoffs procedures, and to study consequences of the resulting analysis. 
We summarize below our main results. 

\begin{enumerate}[label=\textbf{\Roman*}., wide, labelwidth=!, labelindent=0pt]

\item \textbf{An extension of AMP theory}. 
In the Model-X knockoffs framework the statistic used for ranking the variables will involve both $\widehat{\beta}_j(\gamma,\lambda)$ and its knockoff counterpart. 
Therefore, we need to study aspects of their joint distribution, rather than just the marginal distribution of $\widehat{\beta}_j(\gamma,\lambda)$ as in \cite{weinstein2017power}. 
To accommodate this, we present a technical extension of existing AMP results, which underlies our analysis, but  may be of independent interest and have broader implications. 
The challenge is, in essence, to extend the convergence results in \cite{bayati2011dynamics} so they apply to functions (of the regression coefficients and their estimates) which are {\it not} symmetric with respect to all variables $1,...,p$. 
This basic result is formulated in Theorem \ref{thm:contrast_gen} (Section~\ref{sec:power}), and in turn facilitates calculation of the power curves in Corollaries \ref{cor:tinf} and \ref{cor:tinf}. 

\begin{table}[]
\centering
\begin{tabular}{| l |c|c|}
\hline
                   & \textbf{selection by Lasso} & \textbf{selection by thresholding} \\ \hline
\rule{0pt}{2.75ex} \textbf{Oracle setup}    & Su et al (2017)             & Wang et al (2020)                  \\ \hline
\rule{0pt}{2.75ex} \textbf{Knockoffs setup} & Weinstein et al. (2017+)    & \textit{current paper};\ \ Wang and Janson (2021)                \\ \hline
\end{tabular}
\caption{Current paper in the context of related works analyzing variable selection in the approximate message-passing setting.}
\label{tab:context}
\end{table}

\item \textbf{Asymptotic power predictions for Model-X knockoffs}.
To give an example of the consequences of the theoretical analysis in the current paper, the right panel of Figure \ref{fig:takeaway} shows asymptotic FDP versus TPP as predicted by the theory for Lasso \eqref{eq:lasso-var-selection} and thresholded-Lasso \eqref{eq:thresh-lasso} in both the oracle and knockoff versions. 
Here the undersampling ratio $\delta\equiv \lim(n/p) = 1$, the noise level $\sigma=1$, and $\beta_j$ has a mixture distribution of point mass at $M = 4.3$ with probability $\epsilon=0.1$, and at zero with probability $0.9$. 
In the figure broken lines represent the oracle procedures, and solid lines correspond to knockoff procedures.  
For thresholded-Lasso curves are shown for both Model-X knockoffs (as proposed in the current paper, and depicted in solid black in the figure) and counting knockoffs with $r=p$ fake columns (solid red). 
For Lasso selection, predictions with Model-X knockoffs are actually harder to obtain, because \eqref{eq:lasso-max} is not as useful an approximation when $W$-statistics are considered, so only the curve for counting knockoffs is shown (solid grey). 
Importantly, the ``oracle" version of thresholded-Lasso is implemented here with the optimal value for $\lambda$, see the discussion in Section \ref{sec:cv}. 
For the Model-X knockoffs version of thresholded-Lasso, the value of $\lambda$ used here is the limit of the (10-fold) cross-validation estimate, denoted later by $\lambdacvinfty$. 

Comparing first the two oracles, it is clear that thresholded-Lasso has a significantly better tradeoff curve: for example, FDP is about 25\% by the time Lasso detects 80\% of the signals, whereas thresholded-Lasso is able to detect about 90\% of the signal with the same FDP. 
Turning to the knockoff procedures, it can be seen that counting knockoffs performs slightly better than Model-X; the reason is that counting knockoffs use the lasso coefficient size itself instead of the difference $W_j$, but this is a small price to pay for Model-X in return for a much more general method. 
More importantly, both knockoff versions for thresholded-Lasso perform substantially better than the knockoffs version of Lasso, in fact much better than the {\it oracle} version for  Lasso, and even the universal lower bound of \cite{su2017false} on FDP (see the next paragraph). 
For example, knockoffs still attains TPP of about 80\% with $\textnormal{FDP}$ just above 10\%.

\item \textbf{Thresholded-Lasso selection breaks through the power-FDR tradeoff diagram of \cite{su2017false} also when knockoffs are used for calibration}. 
In \cite{su2017false} a power-FDR tradeoff diagram is provided for Lasso selection, which specifies the upper limit on the asymptotic power subject to maintaining $\fdr\leq q\in (0,1)$; this diagram depends on the undersampling ratio $\delta=n/p$ and the sparsity $\epsilon$, but holds independently of the magnitude of the nonzero regression coefficients. 
A consequence of the results in \cite{su2017false} is that, when the sparsity $\epsilon >0$, Lasso selection will fail to exactly recover the true model. 
In a recent article \cite{TardivelBogdan2018} proved that the aforementioned tradeoff diagram does not apply to thresholded-Lasso selection. 
More specifically, it is shown that for any value of the tuning parameter $\lambda$, proper thresholding of the Lasso coefficient estimates identifies the true model as long as the signal is strong enough, and provided $\epsilon<\rho(\delta)$, where $\rho(\cdot)$ is the famous Donoho-Tanner phase transition curve. 
However, the appropriate threshold depends on unknown parameters such as the sparsity $\epsilon$ and the signal magnitude, hence the practical significance of the results in  \cite{TardivelBogdan2018} is limited. 
In Theorem \ref{thm:break} we give a more quantitative result for the setting considered in the current paper, proving that a Model-X knockoffs analog of the thresholded-Lasso procedure still breaks through the tradeoff diagram of \cite{su2017false}. 
Thus, for any $q\in(0,1)$ Model X-knockoffs equipped with the Lasso coefficient-difference \cite[LCD hereafter]{candes2018panning} statistic, achieves power arbitrarily close to 1 if the signal is strong enough and the sparsity is below the transition curve corresponding to the augmented design.

\item \textbf{Optimal $\lambda$ is well approximated by cross-validation}. 
For a fixed $q$, the performance in terms of achievable TPP of the oracle thresholding selection procedure \eqref{eq:thresh-bridge} that has (asymptotic) $\text{FDP}$ level $q$, in general depends strongly on $\lambda$. 
For thresholded-Lasso ($\gamma=1$) this is demonstrated in \cite{wang2017bridge}, where a characterization is also given for the value of $\lambda$ that asymptotically maximizes TPP for a prescribed FDP level. 
When incorporating knockoffs, the analysis is more subtle because we operate with the difference in the estimate size between a variable and its knockoff counterpart, instead of the estimates themselves. 
While the dependence of the {\it exact} optimal $\lambda$ on the unknown parameters of the problem is fairly complicated, we demonstrate that, at least in the case of i.i.d.\ $\bX$, the optimal $\lambda$ can be well estimated by cross-validation on the augmented design. 
 To allow incorporating this into our asymptotic predictions, in Section \ref{sec:cv} we derive the formula for the limiting value of $\lambda$ chosen by cross-validation on the augmented design. 

\end{enumerate}

\begin{figure}
  \centering
\includegraphics[scale=.45]{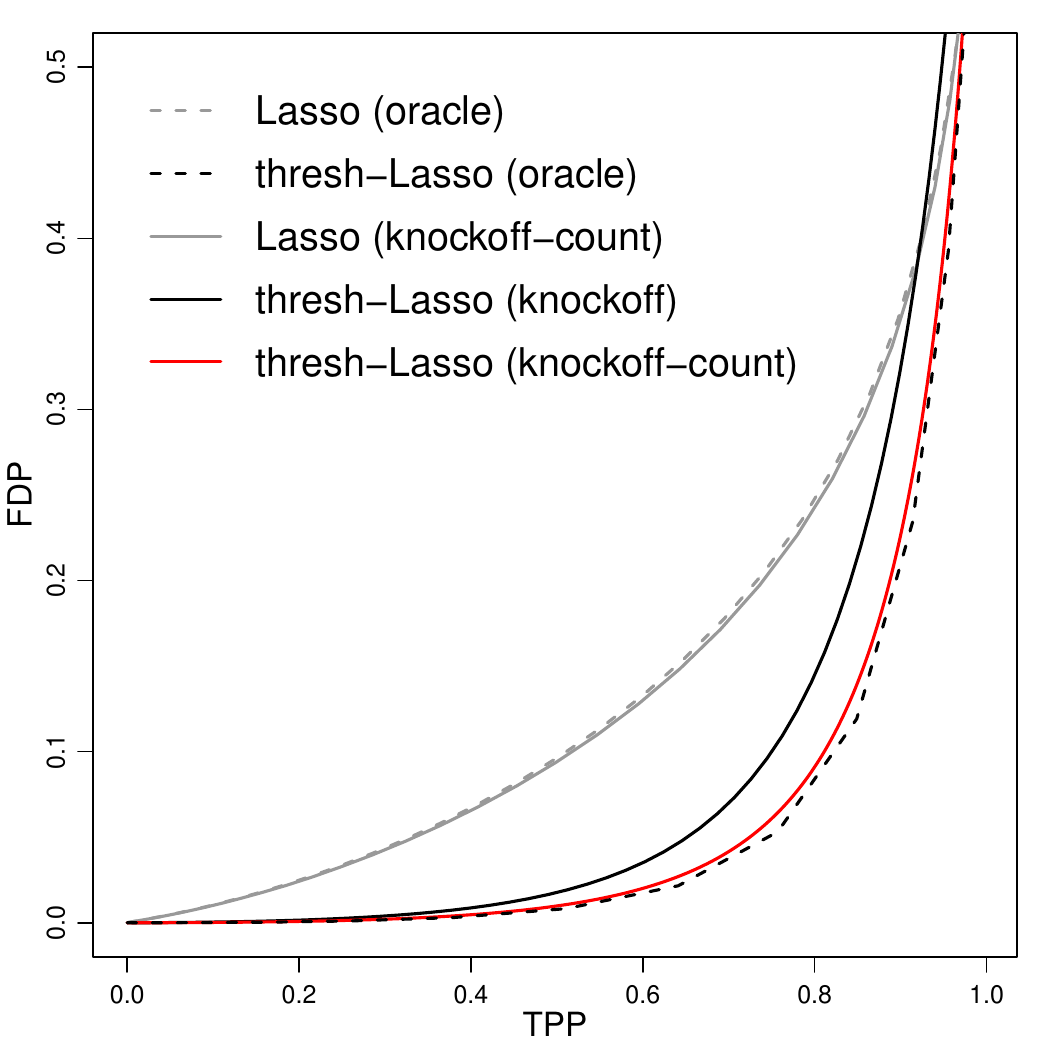}
\caption{
Asymptotic predictions for type-I error vs.~power in the i.i.d.~Gaussian $\bX$ setting, Lasso vs. thresholded-Lasso. 
See details in the main text.
}
\label{fig:takeaway}
\end{figure}

\section{Setup and review}\label{sec:section2}
\subsection{Setup}\label{subsec:setting}
Adopting the basic setting from \cite{su2017false}, our working hypothesis entails the linear model \eqref{eq:model-linear} with $\sigma^2$ fixed and unknown, and we consider an asymptotic regime where $n,p\to \infty$ such that $n/p\to \delta>0$. 
We assume that the matrix $\bX$ has i.i.d.\ $\mathcal{N}(0, 1/n)$ entries, so that the columns are approximately normalized. 
The components $\beta_j$ of the coefficient vector $\beta$ are assumed to be i.i.d.~copies of a mixture random variable, 
\begin{equation}\label{eq:Pi}
\Pi = (1-\epsilon)\delta_0 + \epsilon \Pi^*,
\end{equation}
where $\epsilon\in (0,1)$ is a constant, and where $\E \Pi^2 < \infty$. 
Here $\P(\Pi^* \neq 0) = 1$, so that $\P(\Pi \ne 0) = \epsilon \in (0,1)$. 
With some abuse of notation, we use $\Pi, \Pi^*$ to refer to either the random variable or its distribution, but the meaning should be clear from the context. 
Other than having a mass at zero, $\Pi$ is completely unknown, which is to say that $\epsilon$ and $\Pi^*$ are unknown. 
Finally, $\bX, \beta$, and $\xi$ are all independent of each other.

\medskip
Many selection rules first use the observed data to order the $p$ variables, that is, for some function $g$, an ``importance" statistic
\begin{equation*}
T = (T_1,...,T_p)^\top = g(\bX,Y)\in \RR^p
\end{equation*}
is computed, where larger (say) values of $T_j$ presumably indicate stronger evidence against the null hypothesis that $\beta_j= 0$. 
We assume that $g$ has the natural symmetry property that if $\bX'$ is obtained from $\bX$ by rearranging the columns, then $g(\bX',y)$ rearranges the elements of the vector $g(\bX,y)$ accordingly.
\footnote{Formally, the requirement is that for any permutation $\pi$ on $(1,...,p)$, $g(\bX_{\pi},y) = [g(\bX,y)]_{\pi}$, where $\bX_{\pi}$ is defined to be the matrix with its $j$-th column equal to the $\pi(j)$-th column of $\bX$. 
This mild condition is needed also in \cite{candes2018panning}.}
Given a target FDR level $q$, a final model can then be selected by taking
\begin{equation}\label{eq:threshold}
\widehat{\mathcal{S}} = \{j: T_j \geq t\},
\end{equation}
where $t=t(q)$ is a threshold that may generally depend on the observed data. 
For any choice of the importance statistic $T$ (i.e., for any choice of $g$), we define
\begin{equation}\label{eq:fdp:tpp:process}
\fdp(t) \equiv \frac{|\{j\in \calH: T_j\geq t, j\in \calH_0\}|}{|\{j\in \calH:T_j\geq t\}|}, \ \ \ \ \ \ \ \tpp(t) \equiv \frac{|\{j\in \calH:T_j\geq t, j\notin \calH_0\}|}{|\{j\in \calH: j\notin \calH_0\}|},
\end{equation}
again with the convention 0/0 = 0. 
In the rest of the paper we will consider importance statistics that derive from the convex program \eqref{eq:beta-bridge}. 
As presented in the Introduction, the case of $\gamma=1$ in the bridge optimization program \eqref{eq:beta-bridge} is of particular interest here, because 
in the Lasso case the estimator $\widehat{\beta}(\lambda)$ itself is sparse and can be used directly for variable selection. 
Therefore, we generally focus on the case $\gamma=1$ from now on, but, importantly, our new results are stated for any $\gamma\geq 1$ to match the generality in  \cite{wang2017bridge}.

\subsection{Basic AMP predictions}\label{subsec:amp}

For the Lasso program \eqref{eq:lasso}, we start with noting that, on defining
\begin{equation}\label{eq:lasso-max}
T_j = \max\{\lambda: \widehat{\beta}_j(\lambda)\neq 0\},
\end{equation}
we have $|\{j:T_j\geq t\}|\approx |\{j:\widehat{\beta}_j(t)\neq 0\}|$, because only variables that drop out from the Lasso path---that is, for which $\widehat{\beta}_j(\lambda_0)\neq 0$ but $\widehat{\beta}_j(\lambda_1)= 0$ for $\lambda_1<\lambda_0$---can contribute to the difference between the quantities; see discussion in \cite{weinstein2017power}. 
%
%
Therefore, we treat the comparison between \eqref{eq:lasso-var-selection} and \eqref{eq:thresh-lasso} as essentially a comparison between two procedures of the form \eqref{eq:threshold}, where $T_j$ is given by \eqref{eq:lasso-max} for Lasso, and by 
\begin{equation}\label{eq:lasso-coeff}
T_j = |\widehat{\beta}_j(\lambda)|
\end{equation}
for thresholded-Lasso. 
In anticipation of Section \ref{sec:power}, we call \eqref{eq:lasso-max} the {\it Lasso-max} statistic, and we call \eqref{eq:lasso-coeff} the {\it Lasso-coefficient} statistic. 

\smallskip
Remarkably, under the working hypothesis, exact asymptotic predictions of FDP and TPP can be obtained for both Lasso and thresholded-Lasso. 
Stated informally, Theorem 1 in \cite{bayati2011dynamics} asserts that {under our modeling assumptions}, in the limit as $n,p\to \infty$ we can ``marginally'' treat
\begin{equation}\label{eq:amp}
\left(\widehat{\beta}_j, \beta_j\right) \stackrel{\cdot}{\sim} \left(\eta_{\alpha \tau}(\Pi + \tau Z), \Pi \right), 
\end{equation}
and we use a dot above the ``$\sim$'' symbol to indicate that this holds only in that restricted sense. 
Above, $\eta_{\theta}(x)\equiv \sgn(x)\cdot (|x|-\theta)_+$ is the soft-thresholding operator (acting coordinate-wise); $Z\sim \mathcal{N}(0,1)$ and independent of $\beta$; and $\tau>0,\ \alpha >\max\{\alpha_0, 0\}$ is the unique solution to 
\begin{equation}\label{eq:system}
\begin{aligned}
\tau^2 &= \sigma^2 + \frac{1}{\delta}\EE(\eta_{\alpha\tau}(\Pi + \tau Z) - \Pi)^2 \\
\lambda &= \left( 1 - \frac{1}{\delta}\PP(|\Pi + \tau Z|\geq \alpha\tau) \right) \alpha\tau.
\end{aligned}
\end{equation}
Furthermore, $\alpha_0$ is the unique root of the equation
$(1+t^2)\Phi(-t) - t\phi(t) = \delta/2$. 
This result underlies the analysis in \cite{su2017false}, where it is formally shown (Lemma A.1) that
\begin{equation}\label{eq:amp:su}
\begin{aligned}
\frac{|\{j:\widehat{\beta}_j(\lambda)\neq 0, j\in \calH_0\}|}{p}&\stackrel{\PP}{\longrightarrow}2(1-\epsilon)\Phi(-\alpha),\\
\frac{|\{j:\widehat{\beta}_j(\lambda)\neq 0, j\notin \calH_0\}|}{p}&\stackrel{\PP}{\longrightarrow}\PP(|\Pi + \tau Z|\geq \alpha\tau, \Pi \neq 0) = \epsilon \PP(|\Pi^* + \tau Z|\geq \alpha\tau), 
\end{aligned}
\end{equation}
with $\alpha, \tau$ and $Z$ as described above. 
For a general importance statistic $T$, define
\begin{equation*}
\fdpinfty(t) \equiv \lim \fdp(t)\ \ \ \ \ \ \ \ \  \tppinfty(t) \equiv \lim \tpp(t),
\end{equation*}
where the limits are in probability. 
We use special notation for the limiting FDP and TPP corresponding to
the Lasso-max and to the Lasso-coefficient statistics: for the choice of $T_j$ in \eqref{eq:lasso-max} we write $\fdpinfty^{\textnormal{LM}}(t)$ and $\tppinfty^{\textnormal{LM}}(t)$, and for the choice of $T_j$ in \eqref{eq:lasso-coeff} we write $\fdpinfty^{\textnormal{LC}}(t;\lambda)$ and $\tppinfty^{\textnormal{LC}}(t;\lambda)$. 
In \cite{weinstein2017power}, \eqref{eq:amp:su} was used to approximate 
\begin{equation}\label{eq:amp:lasso:max}
\begin{aligned}
\fdpinfty^{\textnormal{LM}}(t) &\approx \frac{2(1-\epsilon)\Phi(-\alpha)}{2(1-\epsilon)\Phi(-\alpha) + \epsilon \PP(|\Pi^* + \tau Z|\geq \alpha\tau)}\\
\tppinfty^{\textnormal{LM}}(t) &\approx \PP(|\Pi^* + \tau Z|\geq \alpha\tau),
\end{aligned}
\end{equation}
where $(\alpha,\tau)$ are the solution to \eqref{eq:system} on replacing $\lambda$ by $t$. 

\medskip
In a more recent work, \cite{wang2017bridge} observed that the implications of \cite{bayati2011dynamics} can, with the necessary adaptations, be used to analyze TPP and FDP also for selection rules of the form \eqref{eq:thresh-bridge}. 
In particular, for thresholded-Lasso, Lemma 2.2 in \cite{wang2017bridge} asserts that
\begin{equation}\label{eq:amp:su:t}
\begin{aligned}
\frac{|\{j: |\widehat{\beta}_j(\lambda)|\geq t, j\in \calH_0\}|}{p} &\stackrel{\PP}{\longrightarrow}2(1-\epsilon)\Phi(-\alpha-t/\tau)\\
\frac{|\{j: |\widehat{\beta}_j(\lambda)|\geq t, j\notin \calH_0\}|}{p} &\stackrel{\PP}{\longrightarrow} \epsilon \PP(|\Pi^* + \tau Z|\geq  t+ \alpha\tau). 
\end{aligned}
\end{equation} 
It then follows that
\begin{equation}\label{eq:amp:lasso:coeff}
\begin{aligned}
\fdpinfty^{\textnormal{LC}}(t;\lambda) &= \frac{2(1-\epsilon)\Phi(-\alpha-t/\tau)}{2(1-\epsilon)\Phi(-\alpha-t/\tau) + \epsilon \PP(|\Pi^*/\tau +  Z|\geq  \alpha + t/\tau)}\\
\tppinfty^{\textnormal{LC}}(t;\lambda) &= \PP(|\Pi^*/\tau +  Z|\geq  \alpha + t/\tau),
\end{aligned}
\end{equation}
where $(\alpha,\tau)$ are determined by $\lambda$ through \eqref{eq:system}. 
Hence, the asymptotic TPP and FDP in \eqref{eq:amp:lasso:coeff} depend on the value of $\lambda$ at which the Lasso estimates are computed. 
Theorem 3.2 in \cite{wang2017bridge} further identifies the asymptotically optimal value of $\lambda$, proving that for any $\lambda>0$, 
\begin{equation*}
\tppinfty^{\textnormal{LC}}(t;\lambda^*) \leq \tppinfty^{\textnormal{LC}}(t;\lambda) \Longrightarrow \fdpinfty^{\textnormal{LC}}(t;\lambda^*) \leq \fdpinfty^{\textnormal{LC}}(t;\lambda)
\end{equation*} 
where 
\begin{equation}
\lambda^* = \displaystyle \argmin_{\lambda} \frac{1}{p}\|\widehat{\beta}(\lambda) - \beta\|^2_2. 
\end{equation} 
By inspection, we see that an equivalent characterization of $\lambda^*$ is the value of $\lambda$ corresponding to the minimum $\tau$ in \eqref{eq:system}. 
This characterization is useful for computing $\lambda^*$ as a function of $\epsilon, \Pi^*, \sigma^2$. 

\smallskip
Comparing the curves 
$
t\mapsto (\tppinfty(t), \fdpinfty(t))
$ 
corresponding to \eqref{eq:amp:lasso:max} and \eqref{eq:amp:lasso:coeff}, \cite{wang2017bridge} concluded that with an appropriate choice of $\lambda$, thresholded-Lasso can improve significantly over Lasso, in the sense that a target TPP level can be achieved with much smaller FDP, and as illustrated by the dotted curves in Figure \ref{fig:takeaway}. 

\subsection{Model-X knockoffs for FDR control}\label{subsec:knockoffs}
The choice of an adequate feature importance statistic is crucial for producing a good ordering of the $\beta_j$'s, from the most likely to be nonnull to the least likely to be nonnull. 
A separate question is how to set the threshold $\hat{t}$ in \eqref{eq:threshold} so that the FDR is controlled at a prespecified level. 
Inspired by \cite{barber2015controlling}, \cite{candes2018panning} proposed a general method for the random-X setting, Model-X knockoffs, that utilizes artificial null variables for finite-sample control of the FDR. 
Assuming that the distribution of the vector $X_i = (X_{i1},...,X_{ip})$ is known (but arbitrary), the basic idea is to introduce, for each of the $p$ original variables, a fake control so that, 
whenever $\beta_j=0$, the importance statistic for the $j$-th variable is indistinguishable from that corresponding to its fake copy. 
This property can then be exploited by keeping track of the number of {\it fake} variables selected as an estimate for the number of false positives. 

\smallskip
Under our working assumptions, the $p$ components $X_{i1},...,X_{ip}$ are i.i.d., in which case the construction of Model-X knockoffs is trivial. 
Thus, let $\ko\in \mathbb{R}^{n\times p}$ be a matrix with \iid~$\mathcal{N}(0, 1/n)$ entries drawn completely independently of $\bX$, $\xi$ and $\beta$, so that it holds in particular that $Y$ and $\ko$ are independent conditionally on $\bX$. 
We refer to $[\bX, \ko]\in \RR^{n\times 2p}$ as the {\it augmented} $X$-matrix. 

\smallskip
Ranking of the original $p$ features is based on contrasting the importance statistic for variable $j$ with that for its knockoff counterpart where, crucially, all importance statistics are computed on the augmented matrix. 
Thus, to obtain the analog of the selection procedure in \eqref{eq:thresh-bridge}, we first compute the $(2p)$-vector given by
\begin{equation}\label{eq:beta-bridge-kf}
\widehat{\beta}(\gamma,\lambda) = \displaystyle \argmin_{b\in \RR^{2p}} \frac{1}{2}\|Y - [\bX, \ko] b\|^2_2 + \lambda \|b\|_\gamma^\gamma
\end{equation}
instead of \eqref{eq:beta-bridge}, and form the differences 
\begin{equation}\label{eq:bridge-coefficient-difference}
W_j = |\widehat{\beta}_j(\gamma,\lambda)| - |\widehat{\beta}_{p+j}(\gamma,\lambda)|, \ \ \ \ \ j=1,...,p. 
\end{equation}
\smallskip
Because $\ko$ is a valid matrix of Model-X knockoffs, we have from Lemma 3.3. in \cite{candes2018panning} that the signs of the $W_j, j\in \mathcal{H}_0$, are i.i.d.\ coin flips (in fact, when $X_{ij},\ j=1,...,p$, are i.i.d., as considered here, this is easy to see directly from symmetry). 
In the knockoffs framework, variables are selected when their $W_j$ is large, that is, 
\begin{equation}\label{eq:threshold-kf}
\widehat{\mathcal{S}} = \{j: W_j\geq \that\},
\end{equation}
where $\that$ is a data-dependent threshold. 
The idea is to rely on the ``flip-sign" property of the $W_j$ to choose $\that$. 
Concretely, applying the knockoff filter by putting
\begin{equation}\label{eq:that}
\that = \min\left\{ t>0: \fdphat (t) \leq q \right\},\ \ \ \ \ \ \ \ \ \ \fdphat (t) \equiv \frac{1+\#\{j:W_j\leq -t\}}{\#\{j:W_j\geq t\}},
\end{equation}
ensures that the selection rule given by \eqref{eq:threshold-kf} controls the FDR at level $q$ by Theorem 3.4 in \cite{candes2018panning}. 

\smallskip
To obtain the knockoffs counterpart of thresholded-Lasso selection, we specialize \eqref{eq:bridge-coefficient-difference} to $\gamma=1$, recovering the {\it Lasso coefficient-difference} (LCD) statistic introduced in \cite{candes2018panning}, 
\begin{equation}\label{eq:lcd}
W_j = |\widehat{\beta}_j(\lambda)| - |\widehat{\beta}_{p+j}(\lambda)|,
\end{equation}
where
\begin{equation}\label{eq:lasso-kf}
\widehat{\beta}(\lambda) = \displaystyle \argmin_{b\in \RR^{2p}} \frac{1}{2}\|Y - [\bX, \ko] b\|^2_2 + \lambda \|b\|_1
\end{equation}
is the Lasso solution for the augmented setup, i.e., the estimator \eqref{eq:beta-bridge-kf} obtained for $\gamma=1$. 
The corresponding selection procedure in \eqref{eq:threshold-kf} will be referred to from now on as the {\it level-$q$ LCD-knockoffs procedure}. 


\smallskip
Similarly to the notation in Section \ref{sec:section2}, we write $\fdpinfty^{\textnormal{LCD}}(t),\ \tppinfty^{\textnormal{LCD}}(t)$, respectively, for $\fdpinfty(t)$ and $\tppinfty(t)$ associated with the statistic \eqref{eq:lcd}. 
Finally, let 
$$
\widehat{\fdpinfty}^{\textnormal{LCD}}(t) \equiv \lim \fdphat (t)
$$
be the limit of the (knockoffs) estimate of FDP given in \eqref{eq:that} for $\gamma=1$.

\smallskip
Before proceeding to the main section, we recall an alternative implementation of knockoffs for the special case of i.i.d. matrices.  

\subsubsection{``Counting" knockoffs for i.i.d.~matrices}\label{subsubsec:counting}
In the special case where $X_{i1},...,X_{ip}$ are i.i.d., there is in fact a simpler approach to implementing a knockoff procedure, as proposed in \cite{weinstein2017power}. 
Instead of pairing each original covariate with a designated knockoff copy ($X_j$ with $\tilde X_j$),
we can leverage the information that the covariates are i.i.d., and therefore exchangeable, to create a single pool of knockoff variables
$\tilde X_1,\dots,\tilde X_r$ that act as a  ``control group'' simultaneously for each $X_1,\dots,X_p$.

To be concrete,
for some integer $r>0$, suppose we make the matrix $\ko$ of dimension $n\times r$ instead of $n\times p$, still with \iid~$\mathcal{N}(0, 1/n)$ entries as before. 
Then by the symmetry in the problem, the distribution of the fitted coefficient vector $\hat\beta_1,\dots,\hat\beta_{p+r}$ (conditional on $\beta$)
is unchanged under any reordering of the indices in the ``extended" null set,
$$
\mathcal{H}_0 \cup \mathcal{K}_0,
$$
where $\mathcal{K}_0 \equiv \{p+1,...,p+r\}$.
This is a stronger notion of exchangeability (all null covariates are exchangeable with all knockoff variables),
as compared to the pairwise exchangeability property of the general Model-X framework (where each null $X_j$ is only exchangeable with its own knockoff copy $\tilde X_j$).
Exploiting this stronger form of exchangeability,
  \cite{weinstein2017power} prove FDR control---for example, we could take the procedure that rejects $H_{0j}$ whenever $\hat\beta_j\geq \hat{t}$ for\footnote{
We note that, while \cite{weinstein2017power} focus on a different statistic, all of their results concerning FDR control apply equally well to what we call the Lasso-coefficient statistic in the following section. 
  }
\begin{equation}
\hat{t} = \inf\left\{ t\in \RR: \frac{\frac{1}{r+1}\sum_{j\in\mathcal{K}_0}\mathbf{1}\{\hat\beta_j>t\}}{\frac{1}{p}\sum_{j\in\mathcal{H}}\mathbf{1}\{\hat\beta_j > t\}} \leq q\right\},
\end{equation}
and use AMP machinery to derive the appropriate formulas for the power. In particular, power is gained from the fact that, 
if we choose $r$ to be smaller than $p$ (e.g., $r=c\cdot p$ for some $0 < c < 1$), 
the variable selection accuracy of the Lasso is better since we have $n$ observations and $p+r = p(1+c)$ many covariates, rather than $n$ observations and $2p$ covariates
as with Model-X knockoffs.

\smallskip
However, the counting knockoffs strategy is extremely specific to the i.i.d.~design setting: if the $X_j$'s are not themselves i.i.d.~(or exchangeable),
then we cannot hope to construct a single control group that can be shared by a heterogeneous set of covariates.
The Model-X construction, with knockoff $\widetilde X_j$ designed to pair with $X_j$, is therefore substantially more interesting to study 
in terms of understanding the performance of this methodology in non-i.i.d.~settings. 



%

\section{AMP predictions for knockoffs} \label{sec:power}
The results presented thus far are not novel. 
In this section we find the asymptotic FDP and TPP for the Model-X knockoffs versions of the thresholded-bridge selection rules \eqref{eq:thresh-bridge}, in particular for the level-$q$ LCD-knockoffs procedure, and present new results. 
For the knockoffs procedure to control the FDR, the i.i.d.~Gaussian assumption on the $p$ coordinates of $X_i = (X_{i1},...,X_{ip})$ is by no means necessary, and there is indeed no such assumption in \cite{candes2018panning}. 
In the current paper, on the other hand, the goal is to compare the (asymptotic) power of the ``oracle" thresholded-bridge selection procedure \eqref{eq:thresh-bridge} to that of its knockoffs version. 
In particular, for $\gamma=1$ we ultimately want to compare the curves
\begin{equation}\label{eq:curves}
q\mapsto \tppinfty^{\textnormal{LC}}(t^{\infty}(q))\ \ \ \ \ \text{vs.}\ \ \ \ \  q\mapsto \tppinfty^{\textnormal{LCD}}\left( \that^{\infty}(q) \right),
\end{equation}
where the quantities $t^{\infty}(q)$ and $\that^{\infty}(q)$ are defined, respectively, as the values $t^{\infty}$ and $\that^{\infty}$ for which
\begin{equation}\label{eq:tinfty}
\fdpinfty^{\textnormal{LC}}\left(t^{\infty}\right) = q,\ \ \ \ \ \ \ \ \ \ \widehat{\fdpinfty}^{\textnormal{LCD}}(\that^{\infty}) = q. 
\end{equation}
Of course, how the two curves in \eqref{eq:curves} compare on power at every given $q$, depends on the underlying model, including the dependence structure among the coordinates of $X_i$. 
We now proceed to obtaining power predictions for Model-X knockoffs under the asymptotic setting of Section \ref{subsec:setting}. 

\medskip 
The main technical challenge is to validate that the theory from \cite{bayati2012} carries over to the {\it knockoff} setup involving $W$-statistics. 
To overcome this technical challenge, we develop a ``local'' version of AMP theory that applies to the broad class of knockoff-calibrated selection procedures in \eqref{eq:threshold-kf}. 
More specifically, as compared to \eqref{eq:amp}, in order to analyze the knockoffs selection procedure \eqref{eq:threshold-kf} we need to study the {triples} $(\widehat{\beta}_{j}(\gamma,\lambda), \beta_j, \widehat{\beta}_{p+j}(\gamma,\lambda))$ rather than the {pairs} $(\widehat{\beta}_{j}(\gamma,\lambda), \beta_j)$. 
Theorem~\ref{thm:contrast_gen} below asserts that, for our asymptotic FDP and TPP calculations, we can treat
\begin{equation}\label{eq:amp-kf}
\left(\widehat{\beta}_{j}(\gamma,\lambda), \beta_j, \widehat{\beta}_{p+j}(\gamma,\lambda) \right) \stackrel{\cdot}{\sim} \left(\eta_{\alpha'\tau'^{2-\gamma}, \gamma}(\Pi + \tau' Z), \Pi, \eta_{\alpha'\tau'^{2-\gamma},\gamma}(\tau' Z')\right),
\end{equation}
which is an extension of \eqref{eq:amp}. 
Above, $Z$ and $Z'$ are independent $\mathcal{N}(0,1)$ random variables that are furthermore independent of $\beta_j$, the operator $\eta_{\theta, \gamma}$ with threshold level $\theta > 0$ is defined as
\begin{equation}\label{eq:gen_soft}
\eta_{\theta, \gamma}(u) := \argmin_{z} \frac12 (u - z)^2 + \theta |z|^\gamma,
\end{equation}
and $(\alpha',\tau')$ are the unique solution to the equation~\cite{weng2018overcoming}
\begin{equation}\label{eq:system-kf-bridge}
\begin{aligned}
&\tau^2 = \sigma^2 + \frac1{\delta} \E\left[\eta_{\alpha\tau^{2-\gamma}, \gamma}(\Pi + \tau Z) - \Pi \right]^2 + \frac1{\delta} \E \eta_{\alpha\tau^{2-\gamma}, \gamma}(\tau Z)^2\\
&\lambda = \left[1 - \frac1{\delta}\E \eta'_{\alpha\tau^{2-\gamma}, \gamma}(\Pi + \tau Z) - \frac1{\delta}\E \eta'_{\alpha\tau^{2-\gamma}, \gamma}(\tau Z) \right] \alpha\tau^{2-\gamma}.
\end{aligned}
\end{equation}
In the special case $\gamma = 1$, where the bridge estimator is just the Lasso estimator, the operator $\eta_{\theta, \gamma}$ reduces to the soft-thresholding operator $\eta_{\theta, 1}(u) = \eta_{\theta}(u)\equiv \sgn(x)\cdot (|u|-\theta)_+$, and \eqref{eq:system-kf-bridge} becomes
\begin{equation}\label{eq:system-kf}
\begin{aligned}
&\tau^2 = \sigma^2 + \frac1{\delta} \E\left[\eta_{\alpha\tau}(\Pi + \tau Z) - \Pi \right]^2 + \frac1{\delta} \E \eta_{\alpha\tau}(\tau Z)^2\\
&\lambda = \left[1 - \frac1{\delta}\P(|\Pi + \tau Z| > \alpha\tau) - \frac1{\delta}\P(|\tau Z| > \alpha\tau)\right] \alpha\tau.
\end{aligned}
\end{equation}
The following theorem formalizes the notion in which \eqref{eq:amp-kf} holds, and is our main theoretical result.


\begin{theorem}\label{thm:contrast_gen}
Let $f$ be any bounded continuous function defined on $\R^3$. Then, we have
\[
\frac1{p}\sum_{i=1}^p f\left(\widehat{\beta}_{i}(\gamma,\lambda), \beta_i, \widehat{\beta}_{p+i}(\gamma,\lambda)\right) \goto \E f\left(\eta_{\alpha'\tau'^{2-\gamma}, \gamma}(\Pi + \tau' Z), \Pi, \tau'\eta_{\alpha', \gamma}(Z')\right)
\]
in probability. Here $(\alpha', \tau')$ are the unique solution to \eqref{eq:system-kf-bridge}, and $Z$ and $Z'$ are two independent standard normal random variables, which are further independent of $\Pi$. 
\end{theorem}

\begin{remark}
Note that the generalized soft-thresholding operator \eqref{eq:gen_soft} satisfies $\eta_{\alpha' \tau'^{2-\gamma}, \gamma}(\tau' z) = \tau' \eta_{\alpha', \gamma}(z)$ for any $z$.
\end{remark}

\smallskip
For the special Lasso case, $\gamma = 1$, we actually have the stronger result below. 
\begin{proposition}\label{thm:contrast}
Under the assumptions of Theorem~\ref{thm:contrast_gen}, the Lasso estimator $\widehat{\beta}(\lambda)$ satisfies 
\[
\frac1{p}\sum_{i=1}^p f \left(\widehat{\beta}_{i}(\lambda), \beta_i, \widehat{\beta}_{p+i}(\lambda)\right) \goto \E f\left( \sgn(\Pi + \tau' Z) (|\Pi + \tau' Z| - \alpha'\tau')_+, \Pi, 
\tau'\sgn(Z') (|Z'| - \alpha')_+ \right)
\]
in probability, where $(\alpha', \tau')$ are the unique solution to \eqref{eq:system-kf}. 
Moreover, the convergence in probability is uniform over $\lambda$ in any compact set of $(0, \infty)$.
\end{proposition}

The proofs of Theorem~\ref{thm:contrast_gen} and Proposition~\ref{thm:contrast} are deferred to Appendix~\ref{sec:proof-theorem}. 
For the Lasso case $\gamma=1$, a similar result was obtained in a simultaneous and independent work by \cite[see their Theorem 6 and the corresponding analysis]{wang2021high}. 
There are, however, some differences. 
First, our Theorem \ref{thm:contrast_gen}, of which the first assertion in Proposition~\ref{thm:contrast} is a direct consequence, applies more generally to any bridge estimator with $\gamma\geq 1$. 
Second, the techniques we use in the proof are quite different, and these allow us to establish {\it uniform} convergence in $\lambda$ for the Lasso case in second assertion of Proposition \ref{thm:contrast}. 
The uniform convergence is essential for our results to apply when selecting $\lambda$ by cross-validation, as we recommend in Section~\ref{sec:cv}. 
Proposition \ref{thm:contrast} is also closely related to Corollary 1 in \cite{bayati2011dynamics}, which can be viewed as a ``marginal'' version of the above assertion: in the Model-X knockoffs context, \cite{bayati2011dynamics} implies the convergence of a sum over all pairs $i, j$ such that $1 \le i, j \le 2p$, as opposed to ``diagonal'' pairs $i, p+i$ for $1 \le i \le p$ in Proposition~\ref{thm:contrast} above. 
Corollary 1 in \cite{bayati2011dynamics} then follows by making use of its conditional (hence stronger) counterpart, Proposition~\ref{thm:contrast}. 
More generally, just as Corollary 1 in \cite{bayati2011dynamics} applies to a tuple of any number of indices, Proposition~\ref{thm:contrast} can be readily extended to multiple knockoffs (where several knockoff copies are generated for each original variable). 
This extension would enable a theoretical comparison similar to that presented in the current paper except with multiple knockoffs, and we leave this interesting direction for future research.

\smallskip
Theorem \ref{thm:contrast_gen} allows us to calculate the limits of $\tpp(t)$ and $\fdp(t)$ for the selection path of the $W$-statistic \eqref{eq:bridge-coefficient-difference} for any $\gamma \geq 1$, which includes the LCD statistic as a special case.


\begin{corollary}\label{cor:tpp-fdp-inf}
For fixed $\gamma\geq1$ and $\lambda>0$, consider the variable selection procedure given by \eqref{eq:threshold-kf}. 
Then the asymptotic FDP and TPP at any fixed threshold $t>0$ are, respectively, 
\begin{equation}\label{eq:tppfdpinfty}
\begin{aligned}
\fdpinfty(t) &= \frac{(1-\epsilon) \P ( |\tau' \eta_{\alpha', \gamma}(Z)| - |\tau' \eta_{\alpha', \gamma}(Z')| \ge t )}{ \P ( |\eta_{\alpha'\tau'^{2-\gamma}, \gamma}(\Pi + \tau' Z)| - |\tau' \eta_{\alpha', \gamma}(Z')| \ge t ) }, \\[0.5ex]
\tppinfty(t) &= \P ( |\eta_{\alpha'\tau'^{2-\gamma},\gamma}(\Pi + \tau' Z)| - |\tau' \eta_{\alpha', \gamma}(Z')| \ge t | \Pi \neq 0).
\end{aligned}
\end{equation}
\end{corollary}

%

Moreover, Theorem \ref{thm:contrast_gen} allows us to calculate the limit of the corresponding knockoffs estimate of the FDP. 

\begin{corollary}\label{cor:tinf}
Fix $\gamma\geq1$ and $\lambda>0$. 
Then for any $t>0$, the limit of $\fdphat (t)$ in Equation \eqref{eq:that} is given by
\begin{equation}\label{eq:fdpinfty}
\widehat{\fdpinfty}(t) = \frac{ \P ( |\eta_{\alpha'\tau'^{2-\gamma}, \gamma}(\Pi + \tau' Z)| - |\tau' \eta_{\alpha', \gamma}(Z')| \le -t ) }{ \P ( |\eta_{\alpha'\tau'^{2-\gamma}, \gamma}(\Pi + \tau' Z)| - |\tau' \eta_{\alpha', \gamma}(Z')| \ge t ) )}. 
\end{equation}
\end{corollary}



\begin{remark}
See the proofs of these two corollaries in Appendix~\ref{sec:proof-theorem}. It can be shown that the convergence is uniform in bounded $t$.
\end{remark}

In particular, from Corollaries \ref{cor:tinf} and \ref{cor:tpp-fdp-inf} we can calculate $\tppinfty^{\textnormal{LCD}}\left( \that^{\infty}(q) \right)$, the asymptotic TPP achievable by the level-$q$ LCD-knockoffs procedure: setting $\gamma=1$, for a given $q$ first compute $\that^{\infty}$ as the value of $t>0$ such that 
$$
\widehat{\fdpinfty}(t) = q, 
$$
and then plug it into the second equation in \eqref{eq:tppfdpinfty} to find $\tppinfty^{\textnormal{LCD}}\left( \that^{\infty} \right)$. 
It is easy to verify the relationship 
\begin{equation}\label{eq:overestimate}
\widehat{\fdpinfty}^{\textnormal{LCD}}(t) = \fdpinfty^{\textnormal{LCD}}(t) + \frac{ \epsilon \P ( |\eta_{\alpha'\tau', 1}(\Pi + \tau' Z)| - |\tau' \eta_{\alpha', 1}(Z')| \le -t | \Pi \neq 0) }{ \P ( |\eta_{\alpha'\tau', 1}(\Pi + \tau' Z)| - |\tau' \eta_{\alpha', 1}(Z')| \ge t ) },
\end{equation}
so that $\widehat{\fdpinfty}^{\textnormal{LCD}}(t)$ overestimates $\fdpinfty^{\textnormal{LCD}}(t)$, the actual asymptotic FDP. 
However, the difference between the two is typically very small: because the random variable $|\eta_{\alpha'\tau', 1}(\Pi + \tau' Z)| - |\tau' \eta_{\alpha', 1}(Z')|$ is designed to tend to large values when $\Pi\neq 0$, the second term on the right hand side of \eqref{eq:overestimate} is typically much smaller than $\epsilon$, for example it converges to zero when the magnitude of nonzero elements of $\beta$ increases. 
In other words, using the observable random variable $\widehat{\fdp}(t)$ in \eqref{eq:that} instead of $\fdp(t)$, does not make LCD-knockoffs overly conservative. 
We note that the conservativeness was a nuisance in the (alternative) ``counting" knockoffs implementation in \cite{weinstein2017power}, where an estimate of $\epsilon$ that requires an extra (user-specified) tuning parameter, was incorporated to mitigate the effect. 
Here, conveniently, the use of $W$-statistics obviates the need to estimate $\epsilon$. 

\smallskip
Figure \ref{fig:simulation} shows knockoffs power, $\tppinfty^{\textnormal{LCD}}\left( \that^{\infty}(q) \right)$, against ``oracle" power ,$\tppinfty^{\textnormal{LC}}(t^{\infty}(q))$, when the nominal FDR value $q$ varies. 
We took $\sigma = 1$, and $\Pi$ to be a mixture of mass $0.9$ at zero and mass $0.1$ at $M=4$, while $\delta$ varies in the four panels. 
The tuning parameter $\lambda$ is selected separately for each procedure: for the oracle, this is the optimal $\lambda$ obtained by minimizing the value of $\tau$; for knockoffs, we use the limit $\lambdacvinfty$ of the (10-fold) cross-validation estimate, see Section \ref{sec:cv}. 
We can see that for $\delta\geq 1$, the powers obtained by knockoffs and the oracle are very similar for any $q$. 
When $\delta$ is smaller, the loss of power is more pronounced. 
This is mainly because the Lasso estimate itself has larger variance $\tau$ for small values of $\delta$; see the left panel of Figure \ref{fig:tau_vs}. 
However, for all considered values of $\delta$ 
the relative difference decreases with the power of the oracle (i.e., when $q$ or the magnitude of nonzero elements of $\beta$ increases). 
The dotted lines in Figure \ref{fig:tpp} are obtained by implementing ``counting" knockoffs instead of Model-X knockoffs ($r=1$); 
the power curve is slightly better as compared to Model-X knockoffs because the importance statistic itself is used for each feature rather than the $W$-statistic. 

\begin{figure}
  \centering
\includegraphics[scale=.5]{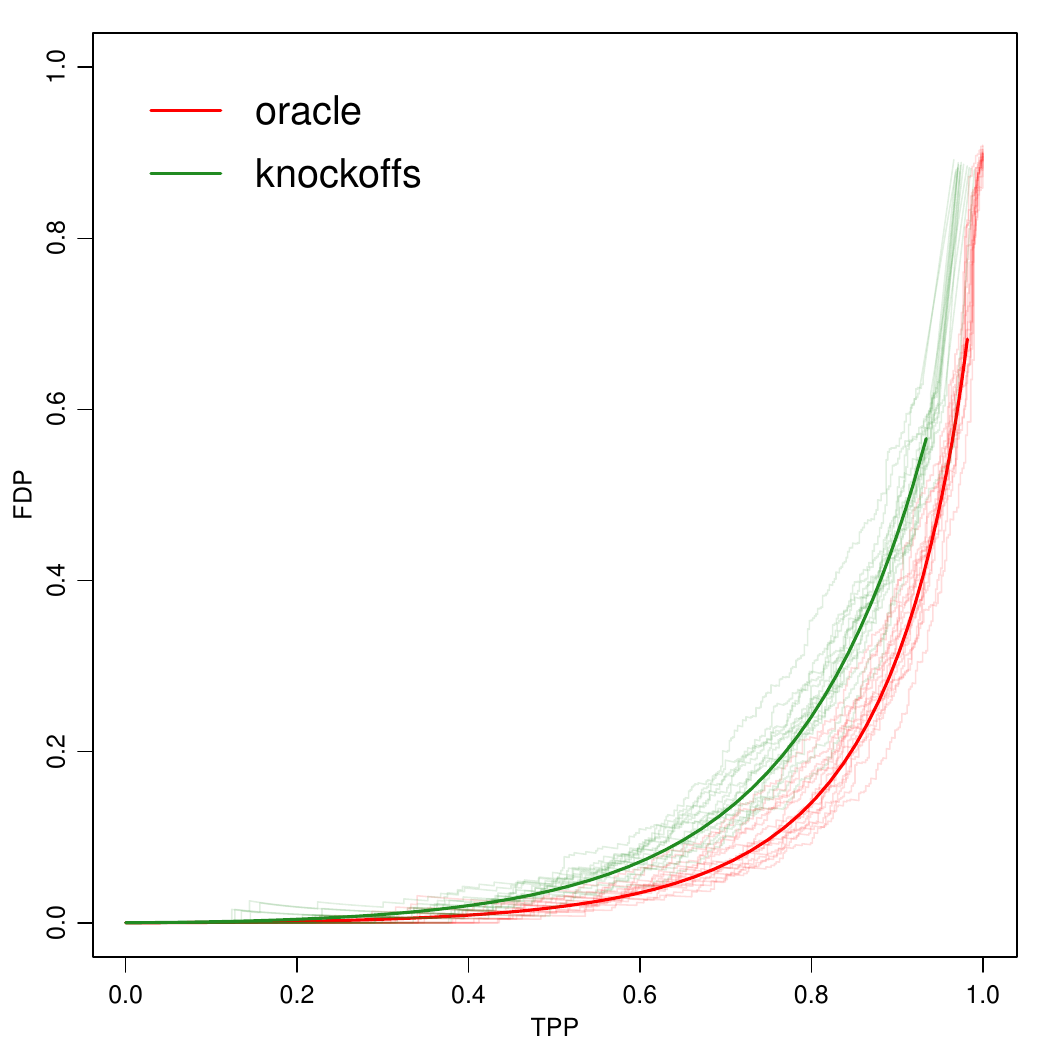} 
\caption{
The right panel presents the FDP-TPP tradeoff curves for thresholded-Lasso in a simulated example. 
Light, thin lines represent (random) realizations from 15 simulated runs. 
Dark, thick lines are theoretical predictions. 
}
\label{fig:simulation}
\end{figure}

\smallskip
Figure \ref{fig:tpp} complements Figure \ref{fig:simulation} by showing FDP-TPP tradeoff paths from a simulation, with the theoretical asymptotic predictions superimposed. 
For both the oracle and the knockoffs versions of thresholded-Lasso we plotted the tradeoff curve in each of 15 realizations from an example with $n=p=5000$ (the other parameters are as in Figure \ref{fig:tpp}). 
To avoid crowding the figure, we plot only the paths for Model-X knockoffs (and not for counting knockoffs). 
We can see a good agreement between the empirical results and the theory.

\begin{figure}
  \centering
\includegraphics[scale=.8]{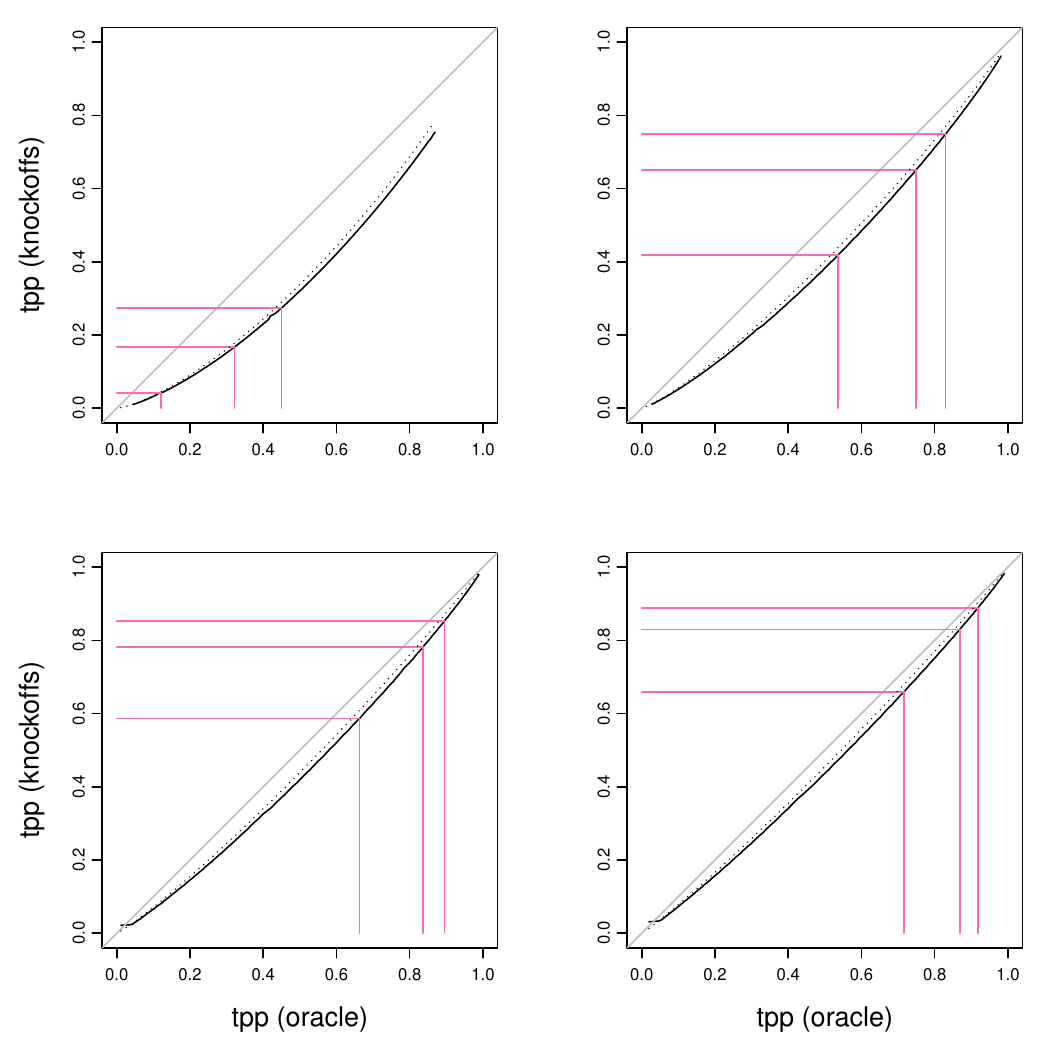} 
\caption{
The parametric curve $q\mapsto \left( \tppinfty^{\textnormal{LC}}(t^{\infty}(q)), \ \tppinfty^{\textnormal{LCD}}\left( \that^{\infty}(q) \right) \right)$. 
For completeness, the analogous curve for the counting knockoffs strategy of Section \ref{subsubsec:counting} with $r=1$, is also shown by the dotted lines. 
Each panel corresponds to a different value of $\delta$: from top left and clockwise, $\delta = 0.5,1,1.5,2$. 
In all panels, $\sigma = 1$ and $\Pi$ has mass $0.9$ at zero and mass $0.1$ at $M=4$. 
Pink segments indicate $q=0.01$ (closest to origin), $0.05$ and $0.1$ (farthest from origin). 
}
\label{fig:tpp}
\end{figure}

\begin{figure}
  \centering
  \begin{subfigure}{.47\textwidth}
  \centering
  \includegraphics[width=\linewidth]{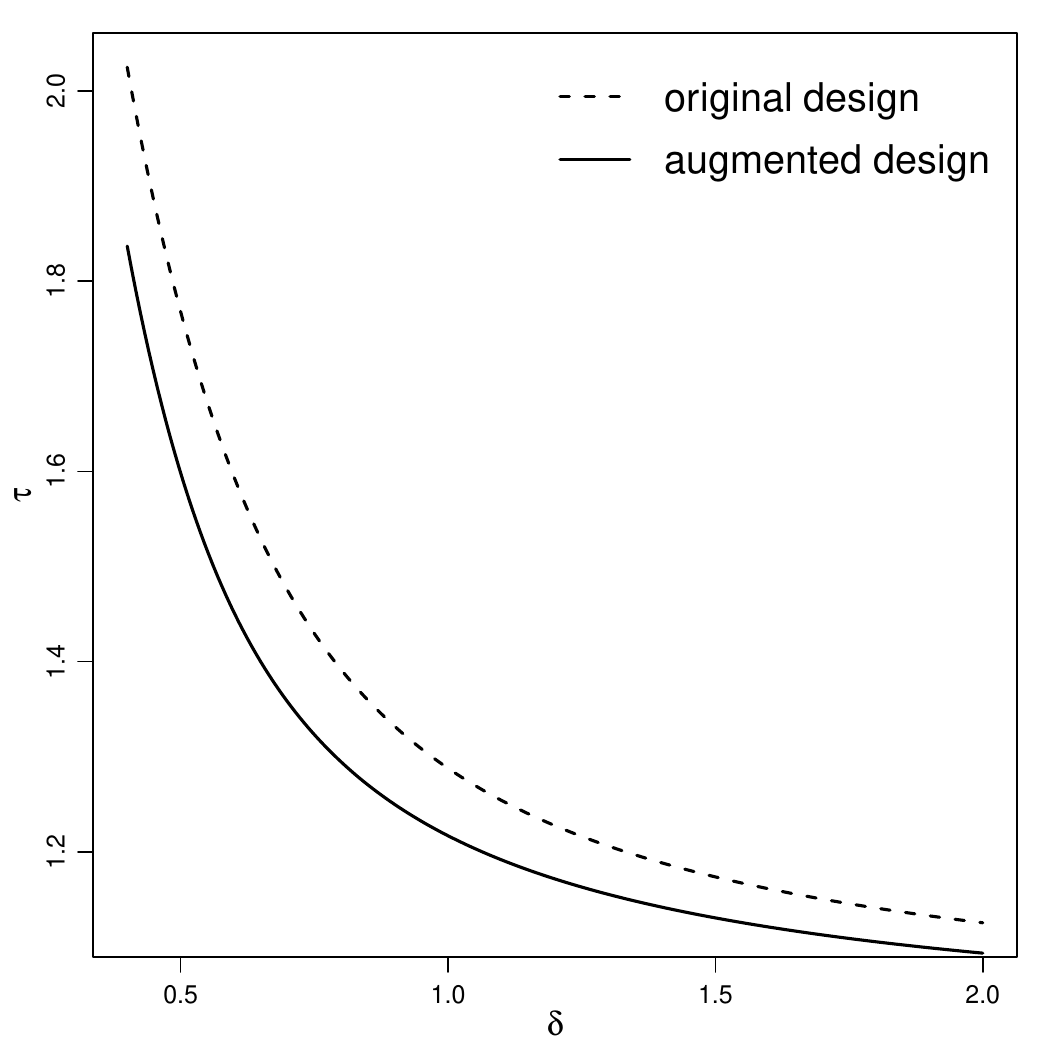}
\end{subfigure}
\begin{subfigure}{.47\textwidth}
  \centering
  \includegraphics[width=\linewidth]{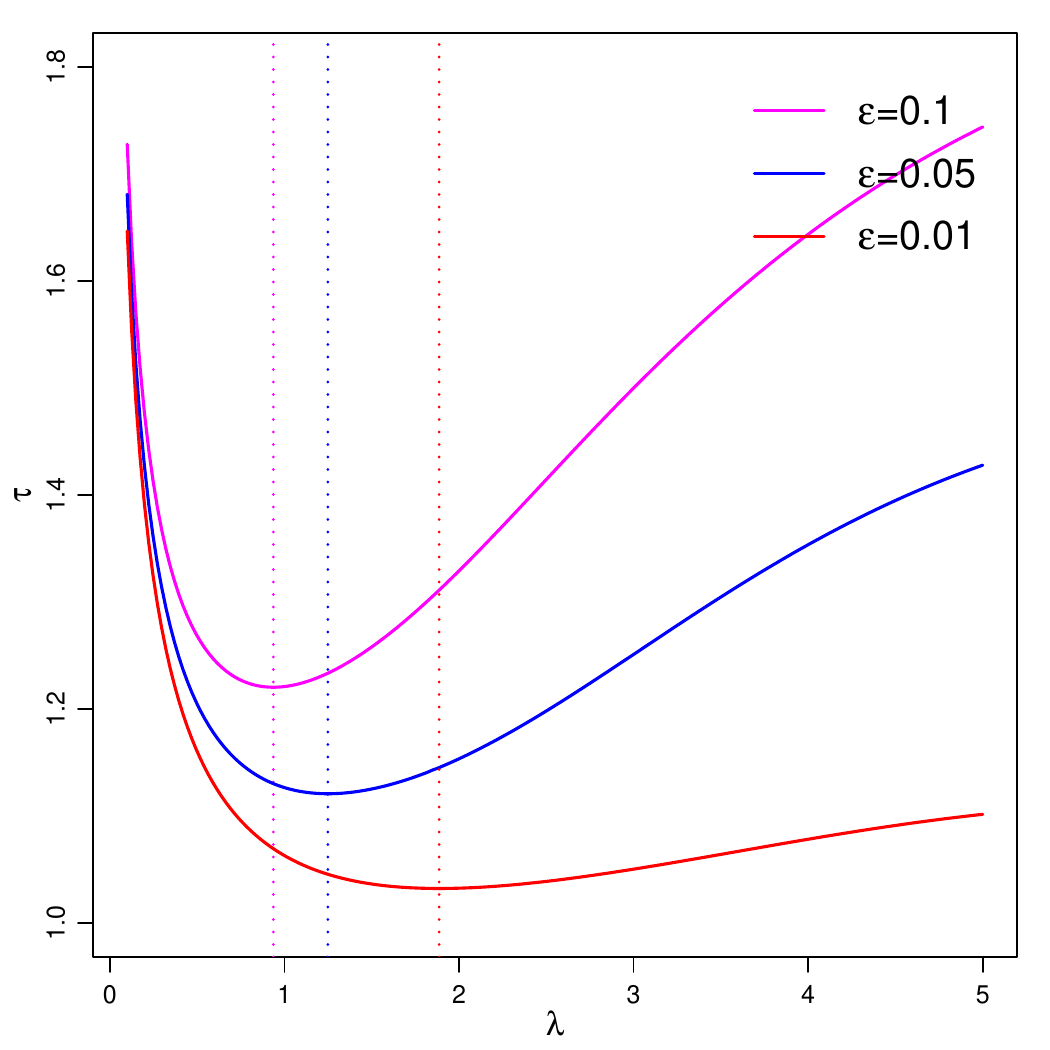}
\end{subfigure}

\caption{The left panel displays the parameter $\tau$ versus $\delta$ for the original and augmented designs, and $\lambda=1$. 
The right panel displays $\tau$ versus $\lambda$ for different values of $\delta$ and $\epsilon$. 
Dotted vertical lines represent the optimal $\lambda$. 
}
\label{fig:tau_vs}
\end{figure}


\medskip
We conclude this section with Theorem \ref{thm:break} below, that applies to the Lasso case $\gamma=1$ and formalizes the notion that the LCD-knockoffs procedure allows to break through the FDP-TPP diagram presented in \cite{su2017false}. 
Specifically, the following result says that for any nominal FDR level $q > 0$ that is not too close to 1, if the signal is strong enough then the LCD-knocknoffs procedure has asymptotic power arbitrarily close to one, as long as the signal sparsity $\epsilon$ satisfies 
\begin{equation}\label{eq:transition}
\epsilon < 2\epsilon^*(\delta/2),
\end{equation}
where $\epsilon^*(\delta)$ is a point on the Donoho--Tanner transition curve \cite{DonTan05}. 

\begin{definition}\label{def:esp-growing}
A sequence of random variables $\Pi_m$ is said to be {\it $\epsilon$-sparse and growing}, if $\P(\Pi_m \ne 0) = \epsilon$ for all $m$, and
\begin{equation*}
\P(|\Pi_m|>M|\Pi_m \neq 0)\rightarrow 1
\end{equation*}
as $m \goto \infty$ for every $M > 0$.
\end{definition}


\begin{theorem}\label{thm:break}
Fix $q>0$ and denote by $\tpp(\lambda, \Pi, q)$ the true positive proportion of the level-$q$ LCD-knockoffs procedure that uses parameter $\lambda$. 
Moreover, fix $\epsilon$ such that \eqref{eq:transition} holds. 
Then for any sequence $\{\Pi_m\}$ that is $\epsilon$-sparse and growing, it holds that for any fixed $0 < \lambda_1 < \lambda_2$ and any $\nu > 0$, there exist $m'$ and $n'(m)$ such that
\[
\P\left ( \inf_{\lambda_1 \le \lambda \le \lambda_2} \tpp(\lambda, \Pi_m, q) > 1 - \nu \right) \ge 1 - \nu
\]
if $m \geq m'$ and $n \ge n'(m)$.
\end{theorem}


\begin{remark}
The proof, which can be found in Appendix \ref{sec:proof-crefthm:break}, shows that this theorem continues to hold for the bridge-based knockoffs procedure that uses $W_j = |\widehat{\beta}_j(\gamma,\lambda)| - |\widehat{\beta}_{p+j}(\gamma,\lambda)|$ with $\gamma>1$. 
For this extension, the nominal level $q$ can take any value between 0 and 1 since the Donoho--Tanner phase transition does not occur for \eqref{eq:beta-bridge-kf} when $\gamma > 1$~\cite{weng2018overcoming}.
\end{remark}

\section{Tuning by cross-validation}\label{sec:cv}
%
%
%
%
The choice of $\lambda$ in the level-$q$ LCD-knockoffs procedure is critical. 
Unlike in the orthogonal $\bX$ situation, the value of $\lambda$ substantially affects the ranking of the variables, because $\lambda$ controls the shrinkage of the Lasso estimates. 
The advantage of the asymptotic theory is that it provides an analytic form for the relationship between $\lambda$ and the parameter $\tau$, so we can use this to characterize a good choice of $\lambda$ by its consequences on the value of $\tau$. 
Figure \ref{fig:tau_vs} below illustrates the dependence of $\tau$ on $\lambda$ for $\delta=1$ and different values of $\epsilon$. 
Here we can see clearly that the relationship is not monotone and that the choice $\lambda\approx 0$ (i.e., recovering the the Basis Pursuit criterion) as well as excessively large $\lambda$ would result in an inflation of the variance of estimates.

\smallskip
Turning to the formal analysis, let $\tppinfty^{\textnormal{LC}}(\lambda) \equiv \tppinfty^{\textnormal{LC}}(t(\lambda); \lambda)$, where $t(\lambda)$ is the smallest positive value such that $\fdpinfty^{\textnormal{LC}}(t(\lambda); \lambda)\leq q$. 
Then Theorem 3.2 in \cite{wang2017bridge} asserts that, for any $q$, 
$$
\lambda\ \text{maximizes } \tppinfty^{\textnormal{LC}}(\lambda)\ \iff\ \lambda\ \text{minimizes } \lim \frac{1}{p} \| \widehat{\beta} - \beta \|^2_2.
$$
In words, the value of $\lambda$ minimizing the asymptotic estimation mean squared error (MSE) is also the optimal $\lambda$ for the testing problem. 
\cite{wang2017bridge} then observe that minimizing the asymptotic MSE, $\EE(\eta_{\alpha\tau}(\Pi + \tau Z) - \Pi)^2$, is in turn equivalent to minimizing $\tau$ in \eqref{eq:system} over $\lambda$. 
Because the minimizer of $\tau$ depends on $\Pi$ and $\sigma$, \cite{wang2017bridge} propose to estimate $\lambda$ in practice by minimizing a consistent {\it estimate} of $\tau$.

%

\smallskip
If the only difference between knockoffs and the oracle were the fact that the augmented $X$-matrix is used instead of the original $X$-matrix, we would be able to conclude immediately that the optimal tuning parameter for LCD-knockoffs is the value of $\lambda$ minimizing $\tau$ in \eqref{eq:system-kf} instead of \eqref{eq:system}. 
This is, however, not the only difference, first because knockoffs use $W$-statistics instead of $\widehat{\beta}$, and secondly because knockoffs utilize an {\it estimate} of FDP instead of the actual FDP in setting the threshold. 
Admittedly, the {\it exact} value of $\lambda$ that is optimal for knockoffs no longer has such a simple characterization, but we can still advocate the $\lambda$ minimizing $\tau$ in \eqref{eq:system-kf} as a good approximation, and this is our target. 
Figure \ref{fig:cv} demonstrates that this approximation is indeed a good one. 

\smallskip
The value of $\lambda$ minimizing $\tau$ in \eqref{eq:system-kf} again depends on the unknown $\Pi$ and $\sigma$. 
To {\it estimate} it, instead of relying on a consistent estimator of $\tau$ as in \cite{wang2017bridge}, we propose to use cross-validation on the augmented design. 
This takes advantage of the fact that when the covariates are i.i.d., minimizing the estimation error is equivalent to minimizing the {\it prediction} error. 
Hence, from now on we write $\lambdacv$ for the $K$-fold cross-validation estimate of $\lambda$ operating on the augmented $X$-matrix. 
We can again predict the exact limit of $\lambdacv$ as follows. 

\begin{lemma}\label{lm:cv}
For fixed $\Pi$, let $\tau(\lambda;\delta)$ be the solution in $\tau$ to \eqref{eq:system-kf} as a function of $\lambda$ and $\delta$. 
Then $\lambdacv$ converges in probability to a constant, call it $\lambdacvinfty$. 
Furthermore, 
\begin{equation}
\lambdacvinfty = \argmin_{\lambda} \tau(\lambda; (K-1)\delta/K),
\end{equation}
where we note that minimizing $\tau$ in \eqref{eq:system-kf} for $\delta,\epsilon, \Pi^*$, is equivalent to minimizing $\tau$ in \eqref{eq:system} for $\delta/2,\epsilon/2, \Pi^*$.
%
\end{lemma}

How to obtain $\lambdacvinfty$ is not immediate from Lemma \ref{lm:cv}: for any value of $\lambda$, $\tau$ is itself given implicitly as the solution to an equation system in two variables, which then needs to be minimized over $\lambda$. 
We can nevertheless define a simple procedure for solving this minimization problem, described in Appendix \ref{appdx:cv-amp} and ultimately yielding the system of equations 
\begin{equation}\label{eq:amp_cv}
\begin{aligned}
&\taucv^2 = \sigma^2 + \frac{K}{(K-1)\delta} \E\left[\eta_{\alphacv\taucv}(\Pi + \taucv Z) - \Pi \right]^2 + \frac{K}{(K-1)\delta} \E [\eta_{\alphacv\taucv}(\taucv Z)]^2\\
&2\phi(\alphacv) - 2\alphacv\Phi(-\alphacv) = \E\left[Z +\alphacv; \Pi + \taucv Z < -\taucv\alphacv \right] - \E\left[Z -\alphacv; \Pi + \taucv Z > \taucv\alphacv \right].
\end{aligned}
\end{equation}
We call \eqref{eq:amp_cv} the CV-AMP equations. 
To obtain $\lambdacvinfty$, we solve the CV-AMP equations, and then use the second equation of \eqref{eq:system-kf} with $(K-1)\delta/K$ substituted for $\delta$ and with $\taucv$ substituted for $\tau$.

\smallskip
Figure \ref{fig:cv} shows power against $\lambda$ for the LCD-knockoffs procedure applied at level $q=0.1$. 
For reference, horizontal lines indicate theoretical power for the knockoffs procedure utilizing the Lasso-max statistic \eqref{eq:lasso-max} (computed on the augmented matrix). 
The latter is obtained from \cite{weinstein2017power} and uses ``counting" knockoffs with the true underlying value of $\epsilon$. 
For LCD, the theoretical predictions are consistent with the simulation results (marker overlays), and demonstrate how drastically power can vary with the choice of the tuning parameter. 
In particular, bad choices of $\lambda$ can lead to smaller power than even the knockoffs version of Lasso \eqref{eq:lasso-var-selection}. 
Vertical solid lines indicate the value of $\lambdacvinfty$, and they indeed seem close to optimal, i.e., close to the value that maximizes power. 
The broken vertical lines represent the simulation average for the $10$-fold cross-validation $\lambda$. 
In accordance with the right panel of Figure \ref{fig:tau_vs}, we can see that the optimal value of $\lambda$ decreases when $\epsilon$ increases. 

\begin{figure}[h]
  \centering
\includegraphics[width=.9\textwidth]{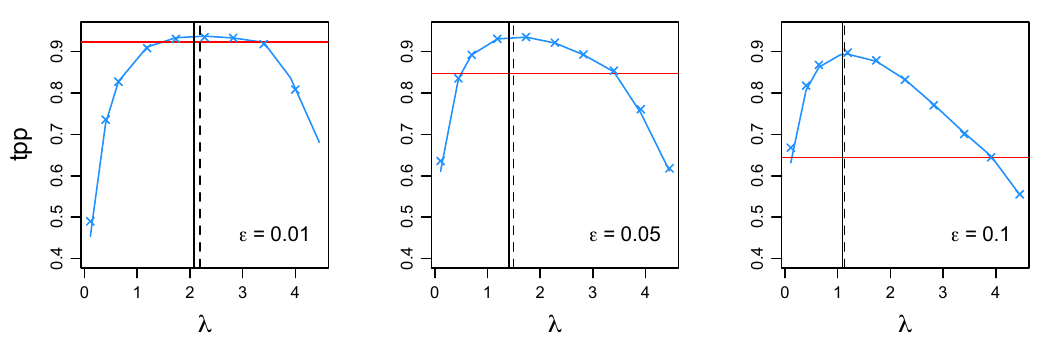} 
\caption{
Power versus $\lambda$ for the level-$q$ LCD-knockoffs procedure, $q=0.1$. 
Light blue curves are theoretical predictions for $\tpp$, marker overlays are averages over $N=100$ simulation runs with $\sigma = 1$, $n=p = 5000$, and $\Pi$ has mass $1-\epsilon$ at zero and mass $\epsilon$ at $M=5$ ($\epsilon$ varies between panels). 
Horizontal red lines indicate predicted TPP for the (``counting") knockoffs procedure using the Lasso-max statistic \eqref{eq:lasso-max}. 
The solid vertical line is the theoretical limit $\lambdacvinfty$, and the broken vertical line is the simulation average, for the cross-validation estimate of $\lambda$ with $K=10$ folds. 
}
\label{fig:cv}
\end{figure}

\smallskip
The boxplots in Figure \ref{fig:box} show sampling variability in $1000$ simulation runs for the cross-validation estimate of $\lambda$ and for the estimate of \cite{wang2017bridge}. 
In all panels we used $n=1000$, $p = 1500$, and $\Pi$ has mass $1-\epsilon = 0.9$ at zero and mass $\epsilon = 0.1$ at $M=5$. 
The red horizontal line indicates $\lambdacvinfty$ for $\delta = n/p=2/3$. 
Sampling variability for cross-validation appears smaller. 
Another (unrelated) advantage of cross-validation is that we have an explicit characterization of $\lambdacvinfty$ through the CV-AMP equations, whereas the analog for the method of \cite{wang2017bridge} is given implicitly as a minimizer of a certain estimate.

\begin{figure}
  \centering
\includegraphics[width=.9\textwidth]{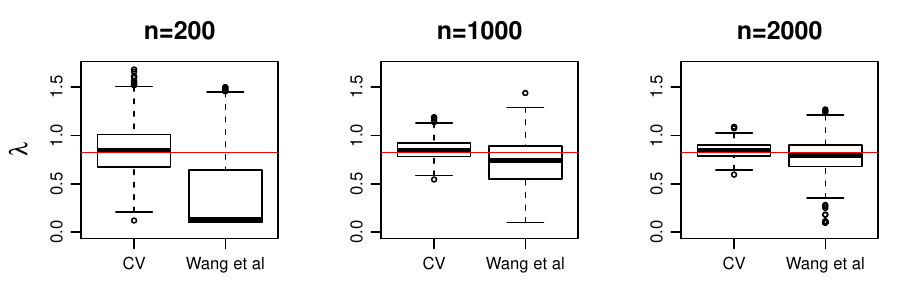} 
\caption{Sampling variability in estimating $\lambda$: CV versus the method of \cite{wang2017bridge}. 
Boxplots are based on $1000$ simulation runs. 
}
\label{fig:box}
\end{figure}

\section{Extension to Type S errors}\label{sec:extension-types}
The classical paradigm, which was also adopted here, regards a predictor as important if the corresponding $\beta_j\neq 0$, and aims at controlling a Type I error rate. 
In practice, however, it is almost always the case that all $\beta_j$ are different from zero to some decimal, in which case the Type I error trivially vanishes.  
In the more general context of multiple comparisons, this has lead to adamant objection to focusing on testing of point null hypotheses \cite{tukey1991philosophy, tukey1960conclusions}. 
A reasonable way out is to consider a predictor as important only if $|\beta_j| \geq \Delta$ for some $\Delta>0$, but this has the disadvantage that the definition depends on $\Delta$. 
Alternatively, Tukey \cite{tukey1991philosophy} advocated procedures that classify the {\it sign} of $\beta_j$ ``with confidence", that is, declare $\beta_j>0$ or $\beta_j<0$ for as many $j$ as possible while keeping small some rate of incorrect decisions on the sign. 
Incorrectly declaring that $\beta_j<0$ when in fact  $\beta_j>0$, or that $\beta_j>0$ when in fact  $\beta_j<0$, is commonly referred to as a {\it Type S} \cite{gelman2000type} or {\it Type III} error. 
For hypothesis testing problems of the type considered in this paper, it is natural to ask what consequences supplementing each rejection with a directional decision has on the error rate. 

\smallskip
As in \cite{barber2019knockoff}, suppose that for each `rejection' $j\in \widehat{\mathcal{S}}$ we further provide an estimate
$$
\widehat{\text{sign}}_j \in \{-1,1\}
$$
of the sign of $\beta_j$. 
We define the {\it false sign proportion} to be
\begin{equation*}
\fsp \equiv \frac{|\{j\in \widehat{\mathcal{S}}: \text{sign}(\beta_j)\neq \widehat{\text{sign}}_j\}|}{|\widehat{\mathcal{S}}|},
\end{equation*}
where $\text{sign}(x)$ is 1,-1 or 0 according as $x>0, x<0$ or $x=0$. 
In particular, we can see that
$$
\fdp\leq \fsp, 
$$
since any false discovery (i.e., selecting $j\in\widehat{\mathcal{S}}$ when in fact $\beta_j=0$) leads to a sign error, $ \text{sign}(\beta_j) = 0\neq \widehat{\text{sign}}_j$.
The false sign proportion may often be much higher than the false discovery proportion---indeed, as demonstrated in \cite{gelman2000type}, 
in a low signal-to-noise regime it is easy for a false discovery rate controlling procedure to have very high false sign rate.

\smallskip
In the sign-classification framework, we can apply our results to obtain exact asymptotic predictions of the FSP and a corresponding notion of power, for the knockoff procedures in our setting. 
Thus, write the nonzero component of the distribution of $\beta_j$ as 
$$
\Pi^* = \pi^{+}\Pi^{+} + \pi^{-}\Pi^{-},
$$
where $\pi^{+} + \pi^{-}=1$ and where $\P(\Pi^{+} > 0) = \P(\Pi^{-} < 0) = 1$. 
Now suppose that for the knockoffs version of the thresholded bridge selection procedure \eqref{eq:beta-bridge}, we further estimate $\widehat{\text{sign}}_j=\text{sign}(\widehat{\beta}_j)$ for each $j\in \widehat{\mathcal{S}}$, where $\widehat{\beta}_j = \widehat{\beta}(\gamma, \lambda)$. 
Taking the Lasso case $\gamma=1$ for example, we can apply Theorem \ref{thm:contrast_gen} to conclude that the procedure that supplements LCD knockoffs with the sign estimates has
\begin{equation}\label{eq:fsp}
\begin{aligned}
&\textnormal{fsp}(t) \equiv \lim \textnormal{FSP}(t) = &&\\[6pt]
& = \frac{
\epsilon \pi^{+} \P(\eta_{\alpha'\tau',1}(\Pi^{+}+\tau'Z)<-|\tau'\eta_{\alpha',1}(Z')|  -t) 
}{
\P(|\eta_{\alpha'\tau',1}(\Pi+\tau'Z)| - |\tau' \eta_{\alpha',1}(Z')|> t)} 
&&+ 
\frac{
\epsilon \pi^{-} \P(\eta_{\alpha'\tau',1}(\Pi^{-}+\tau'Z)> |\tau'\eta_{\alpha',1}(Z')|+ t)}{\P(|\eta_{\alpha'\tau',1}(\Pi+\tau'Z)| - |\tau' \eta_{\alpha',1}(Z')|> t)} && \\[8pt]
& && + \frac{(1-\epsilon) \P(|\tau' \eta_{\alpha',1}(Z)| - |\tau' \eta_{\alpha',1}(Z')|> t)}{\P(|\eta_{\alpha'\tau',1}(\Pi+\tau'Z)| - |\tau' \eta_{\alpha',1}(Z')|> t)}.\\[2pt]
\end{aligned}
\end{equation}

To quantify the power of a procedure in the sign problem, it may at first seem natural to consider the ratio of the number of correctly classified signs divided by the total number of nonzero $\beta_j$'s. 
But, in a regime where there are no exact zeros among the coefficients, this definition is not useful because the denominator will equal $p$, the total number of coefficients, although most (usually, almost all) of the coefficients are still too close to zero in magnitude to be picked up by the selection procedure. 


\smallskip
To overcome this difficulty, we consider a different model, where the distribution of $\beta_j$ is again a mixture between signals (the ``slab'') and nulls (the ``spike''), 
but now the null distribution is concentrated near zero instead of being a point mass at zero.
In particular, let $S_j=0$ and $S_j=1$ indicate if $\beta_j$ is considered to be a ``null'' or ``nonnull'' coefficient, respectively, and
assume the following distribution:
$$
\begin{aligned}
&\PP(S_j = 0) = 1-\epsilon \ \ \ \ \ \ \ &&\PP(S_j = 1) = \epsilon\\[5pt]
&\beta_j| S_j = 0 \ \sim \Pi_0 \ \ \ \ \ \ \ &&\beta_j | S_j = 1 \ \sim \Pi_1, 
\end{aligned}
$$
where, consistent with Tukey's viewpoint mentioned earlier, we will assume that neither of $\Pi_0, \Pi_1$ has point mass at zero, but $\Pi_0$ still represents a ``spike" and is 
concentrated near zero and $\Pi_1$ represents a ``slab" component of the mixture. 
Notice that this ``two group" model entails
\begin{equation}\label{eq:two-groups}
\Pi = (1-\epsilon)\Pi_0 + \epsilon\Pi_1
\end{equation}
as the distribution of $\beta_j$, analogous to \eqref{eq:Pi}. 
The {\it true sign proportion} is then defined as
\begin{equation}\label{eq:tsp-correction}
\tsp \equiv \frac{|\{j\in \mathcal{S}_1: \ \text{sign}(\beta_j)= \widehat{\text{sign}}_j\}|}{|\mathcal{S}_1|}, 
\end{equation}
where $\mathcal{S}_1\equiv \{j:S_j=1\}$. 

Appealing again to Theorem \ref{thm:contrast_gen}, the limit of TSP for the knockoffs sign-classification version of \eqref{eq:thresh-lasso} can be calculated as
\begin{equation}\label{eq:tsp-correction-lim}
\begin{aligned}
&\textnormal{tsp}(t) \equiv \lim \textnormal{TSP}(t) = &&\\[6pt]
& = \pi^+_1 \P(\eta_{\alpha'\tau',1}(\Pi^+_1+\tau'Z)<-|\tau'\eta_{\alpha',1}(Z')|  -t) + 
\pi^-_1 \P(\eta_{\alpha'\tau',1}(\Pi^-_1+\tau'Z) > |\tau'\eta_{\alpha',1}(Z')|  +t),
\end{aligned}
\end{equation}
where $\pi^+_1 = \P(\Pi_1>0), \pi^-_1 = \P(\Pi_1<0)$, and where $\Pi^+_1$ is the conditional distribution of $\Pi_1$ given $\Pi_1>0$ and $\Pi^-_1$ is the conditional distribution of $\Pi_1$ given $\Pi_1<0$. 

\medskip
We turn to discussing asymptotic FSP control under the assumption that $\Pi$ has no point mass at zero. 
Recall that by the definition we use for FSP, which subsumes incorrect rejections of zero coefficients, FSP is formally at least as large as FDP in any setting---but, in this specific
case, we will actually have FDP$\equiv 0$ since there are no exact zeros. 
Nonetheless, we will now show that in our new model~\eqref{eq:two-groups}
 that replaces exact zeros with {\it approximate} zeros, the Model X knockoffs at the nominal FDP level $q$ can control FSP at the level $q/2$. The
 factor of 2 is due to the fact that, in this new model, for any $\beta_j$ we can only err in one direction, while in the idealized model where nulls are exactly zero,
 estimating {\em either} a positive or negative value for $\hat\beta_j$ results in an error.

To show this formally, we will compare two different scenarios: first, we will consider the false {\em sign} rate 
under the model
$$
\beta_j\sim \Pi^{(1)}:= (1-\epsilon)\Pi_0 + \epsilon\Pi_1
$$
where there are no exact zeros as in~\eqref{eq:two-groups}, and second, we will consider the false {\em discovery} rate under the model
$$
\beta_j\sim \Pi^{(2)}:= (1-\epsilon)\delta_0 + \epsilon\Pi_1
$$
where now there {\em are} exact zeros as in~\eqref{eq:Pi}. To avoid confusion between these two distributions, we will write
$\textnormal{fsp}_{\Pi^{(1)}}(t)$ and $\textnormal{fdp}_{\Pi^{(2)}}(t)$ for these two quantities of interest, respectively, to emphasize that we are working with
two different distributions. 
Nevertheless, if the ``spike'' distribution $\Pi_0$ is concentrated extremely close to zero, then 
the two resulting data distributions are essentially indistinguishable, which is why we can compare the two.
Formally, below we will show that
$$
\textnormal{fsp}_{\Pi^{(1)}}(t) \leq \frac{0.5}{1-\epsilon}\cdot \textnormal{fdp}_{\Pi^{(2)}}(t),
$$
which is approximately $0.5\textnormal{fdp}_{\Pi^{(2)}}(t)$ if $\epsilon\approx 0$. 
To make the argument clearer, for the rest of this section we 
denote
$\beta \equiv \Pi$, $\hat{\beta}\equiv \eta_{\alpha'\tau',1}(\Pi + \tau'Z)$, and $\tilde{\beta}\equiv \tau'\eta_{\alpha',1}(Z)$. 
Then 
$$
\begin{aligned}
\textnormal{fsp}_{\Pi^{(1)}}(t) &= 
\frac{
\P_{\Pi^{(1)}}(\beta>0) \P_{\Pi^{(1)}}(\hat{\beta}<-|\tilde{\beta}| - t\big|\beta>0) + \P_{\Pi^{(1)}}(\beta<0) \P_{\Pi^{(1)}}(\hat{\beta} > |\tilde{\beta}| + t\big|\beta<0)
}{\P_{\Pi^{(1)}}(|\hat{\beta}| - |\tilde{\beta}|> t)}\\
&\leq 
\frac{
\P_{\Pi^{(1)}}(\beta>0) \P(\hat{\beta}<-|\tilde{\beta}| - t\big| \beta = 0) + \P_{\Pi^{(1)}}(\beta<0) \P(\hat{\beta} > |\tilde{\beta}| + t\big|\beta = 0)
}{\P_{\Pi^{(1)}}(|\hat{\beta}| - |\tilde{\beta}|> t)}\\
&= 
\frac{
0.5\cdot \P_{\Pi^{(1)}}(|\hat{\beta}| - |\tilde{\beta}|> t\big| \beta = 0)
}{\P_{\Pi^{(1)}}(|\hat{\beta}| - |\tilde{\beta}|> t)}  
\ \ 
=\frac{0.5}{1-\epsilon} \cdot \frac{(1-\epsilon)\P_{\Pi^{(1)}}(|\hat{\beta}| - |\tilde{\beta}|> t\big| \beta = 0)}{\P_{\Pi^{(1)}}(|\hat{\beta}| - |\tilde{\beta}|> t)}\\
&= \frac{0.5}{1-\epsilon} \cdot \frac{\P_{\Pi^{(2)}}(\beta=0)\cdot \P_{\Pi^{(2)}}(|\hat{\beta}| - |\tilde{\beta}|> t\big| \beta = 0)}{\P_{\Pi^{(2)}}(|\hat{\beta}| - |\tilde{\beta}|> t)} \cdot \frac{\P_{\Pi^{(2)}}(|\hat{\beta}| - |\tilde{\beta}|> t)}{\P_{\Pi^{(1)}}(|\hat{\beta}| - |\tilde{\beta}|> t)}\\
&= \frac{0.5}{1-\epsilon} \cdot  \textnormal{fdp}_{\Pi^{(2)}}(t) \cdot \frac{\P_{\Pi^{(2)}}(|\hat{\beta}| - |\tilde{\beta}|> t)}{\P_{\Pi^{(1)}}(|\hat{\beta}| - |\tilde{\beta}|> t)} \leq \frac{0.5}{1-\epsilon} \cdot  \textnormal{fdp}_{\Pi^{(2)}}(t),
\end{aligned}
$$
where in the last step we observe that $\P_{\Pi^{(1)}}(|\hat{\beta}| - |\tilde{\beta}|> t)\leq \P_{\Pi^{(2)}}(|\hat{\beta}| - |\tilde{\beta}|> t)$ for any $\Pi_0$. 
Notice that, since $t$ was arbitrary, this holds also 
 for the asymptotic knockoff threshold $\hat{t}^{\infty}_{\Pi^{(1)}}$.
 Moreover, observe that  by continuity of the formula 
for $\hat{t}^{\infty}$, $\hat{t}^{\infty}_{\Pi^{(1)}}\rightarrow \hat{t}^{\infty}_{\Pi^{(2)}}$ when the ''null'' distribution $\Pi_0$ converges to the point mass at zero.
Thus, for $\epsilon<0.5$ and $\Pi_0$ sufficiently concenrated around 0 it holds
$$
\textnormal{fsp}_{\Pi^{(1)}}\left(\hat{t}^{\infty}_{\Pi^{(1)}}\right)\leq   \textnormal{fdp}_{\Pi^{(2)}}\left(\hat{t}^{\infty}_{\Pi^{(2)}}\right)\leq q\;\;,$$ 
which allows to conclude that knockoffs allow for the asymptotic FSP control under $\Pi^{(1)}$.
\color{black}
%
%


\begin{figure}
  \centering
\includegraphics[scale=.5]{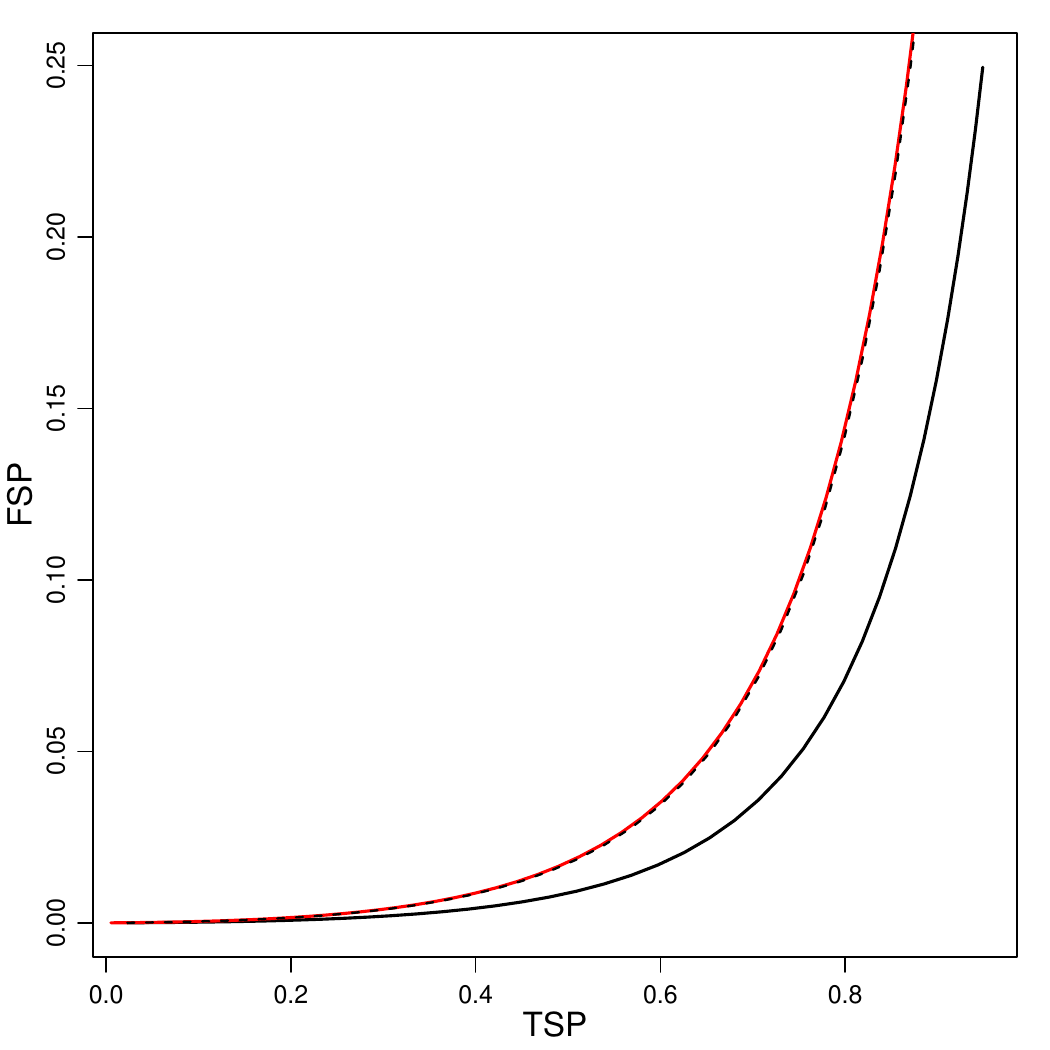}
\caption{
Asymptotic predictions for FSP against TSP in a setting that parallels that of Figure \ref{fig:takeaway}. 
The curved as explained in the main text. 
}
\label{fig:fsp}
\end{figure}

\smallskip
In fact, when $\Pi_0$ is sufficiently concentrated around zero, and $\Pi_1$ sufficiently dispersed, the formulas above will imply
\begin{equation}\label{eq:fsp-vs-fdp}
\textnormal{fsp}_{\Pi^{(1)}}(t) \approx 0.5 \cdot \textnormal{fdp}_{\Pi^{(2)}}(t). 
\end{equation}
%
Figure \ref{fig:fsp} shows the parametric curve $(\textnormal{tsp}_{\Pi^{(1)}}(t), \textnormal{fsp}_{\Pi^{(1)}}(t))$ for a setting with $\delta = 1, \epsilon = 0.1$, $\Pi_0 = \mathcal{N}(0,0.01)$, and $\Pi_1$ is point mass at $M=4.3$. 
The red curve in the figure shows 
$\textnormal{fdp}_{\Pi^{(2)}}(t)$ vs.~$\textnormal{tdp}_{\Pi^{(2)}}(t)$, 
i.e., this is the curve from Figure \ref{fig:takeaway}. 
The relationship \eqref{eq:fsp-vs-fdp} is verified by plotting (dotted black line) the curve $(\textnormal{tsp}_{\Pi^{(1)}}(t), 2\cdot \textnormal{fsp}_{\Pi^{(1)}}(t))$, which essentially coincides with the red curve.

\section*{Acknowledgement}

A.~W.~is supported by ISF via grant 039-9325. 
W.~J.~S.~is partially supported by NSF via grant CCF-1934876, and by the Wharton Dean's Research Fund. 
M.~B.~is supported by the Polish National Center of Science via grant 2016/23/B/ST1/00454. 
R.~F.~B.~is supported by NSF via grant DMS-1654076, and by the Office of Naval Research via grant N00014-20-1-2337. 
E.~C.~is partially supported by NSF via grants DMS 1712800 and DMS 1934578. 

\bibliographystyle{abbrv}
\bibliography{ref}


\appendix

\newcommand{\rina}[1]{\textcolor{teal}{[RFB: #1]}}
\newcommand{\dgoto}{\stackrel{\P}{\longrightarrow}}
\newcommand\ef{f^{\infty}}
\renewcommand{\Longrightarrow}{\stackrel{\P}{\longrightarrow}}

\section{Proofs}
\label{sec:proofs}

We prove Theorems~\ref{thm:contrast_gen} and \ref{thm:break} in this appendix. The proofs rely heavily on some extensions of AMP theory and approximation results for continuous functions, which we first present in Section~\ref{sec:local-amp-lemma} and the beginning of Section~\ref{sec:proof-theorem}, respectively.

\subsection{Local AMP lemmas}
\label{sec:local-amp-lemma}

Following the setting of AMP theory as specified earlier in Section~\ref{sec:section2}, we present some extensions of AMP theory for the Lasso method. We call these results \textit{local} AMP lemmas because these results apply to a subset of the coordinates of the coefficients, unlike the existing AMP results which apply to the entire set of coordinates.


Throughout this appendix, we use $\Longrightarrow$ to denote convergence in probability. 
For simplicity, we also denote $\widehat\beta_{j,\gamma}\equiv \widehat{\beta}_j(\gamma,\lambda)$. 
Recall that $\alpha', \tau'$ are the unique solutions to the set of equations~\eqref{eq:system-kf-bridge}.

\begin{lemma}\label{lm:factored}
Let $g: \R^2 \goto \R$ and $h: \R \goto \R$ be two bounded continuous functions. We have
\[
\frac1{p}\sum_{i=1}^p g(\widehat\beta_{i,\gamma}, \beta_i) h(\widehat\beta_{p+i, \gamma}) \Longrightarrow \E \left[ g(\eta_{\alpha'\tau'^{2-\gamma}, \gamma}(\Pi + \tau' Z), \Pi) \right] \cdot \E \left[h(\tau' \eta_{\alpha', \gamma}(Z)) \right].
\]

\end{lemma}

Lemma~\ref{lm:factored} is the main contribution of this subsection. Its proof relies on the following three lemmas and we defer the proofs of these preparatory lemmas later in this subsection.

\begin{lemma}\label{lm:part_amp}
Let $f: \R \goto \R$ be any bounded continuous function. We have
\[
\frac1{p}\sum_{i=1}^p f(\widehat\beta_{p+i, \gamma}) \Longrightarrow \E f(\tau'\eta_{\alpha', \gamma}(Z)).
\]
\end{lemma}

\begin{lemma}\label{lm:p_part}
Let $f: \R^2 \goto \R$ be any bounded bivariate continuous function. We have
\[
\frac1{p}\sum_{i=1}^p f(\widehat\beta_{i, \gamma}, \beta_{i}) \Longrightarrow \E f(\eta_{\alpha'\tau'^{2-\gamma}, \gamma}(\Pi+\tau' Z), \Pi).
\]
\end{lemma}

\begin{lemma}\label{lm:ab_perm}
For any numbers $A_1, \ldots, A_p$ and $B_1, \ldots, B_p$, denote by $\overbar A$ and $\overbar B$ their respective means. Let $\pi$ be drawn from all permutations of $1, \ldots, p$ uniformly at random. Then, we have
\[
\Var(A_1 B_{\pi(1)} + \cdots + A_p B_{\pi(p)}) = \frac{\left[\sum_{l=1}^p (B_l - \overbar B)^2 \right] \left[\sum_{l=1}^p (A_l - \overbar A)^2 \right]}{p-1}.
\]
\end{lemma}

\begin{proof}[Proof of Lemma \ref{lm:factored}]

By Lemma \ref{lm:part_amp}, we have
\[
\frac1{p}\sum_{i=1}^p   h(\widehat\beta_{p+i, \gamma}) \dgoto \E
\left[  h(\tau'\eta_{\alpha', \gamma}(Z)) \right],
\]
and Lemma \ref{lm:p_part} gives
\[
\frac1{p}\sum_{i=1}^p   g(\widehat\beta_{i, \gamma}, \beta_i)  \dgoto \E
\left[   g(\eta_{\alpha'\tau'^{2-\gamma}, \gamma}(\Pi + \tau' Z), \Pi) \right].
\]
Now, let us consider the distribution of
\begin{equation}\label{eq:gh_prod}
\frac1{p}\sum_{i=1}^p   g(\widehat\beta_{i,\gamma}, \beta_i)   h(\widehat\beta_{p+i, \gamma})
\end{equation}
conditional on $  g(\widehat\beta_{1,\gamma}, \beta_1), \ldots,   g(\widehat\beta_{p,\gamma}, \beta_p)$ and the empirical distribution of $\{h(\widehat\beta_{p+i, \gamma})\}_{i=1}^p$. This $\sigma$-algebra is denoted as $\F$. Note that knowing the empirical distribution of $\{h(\widehat\beta_{p+i, \gamma})\}_{i=1}^p$ is the same as knowing all values of $h(\widehat\beta_{p+i, \gamma})$ except for the indices. 
By symmetry, the conditional distribution of \eqref{eq:gh_prod} is the same as that of
\begin{equation}\nonumber
\frac1{p}\sum_{i=1}^p g(\widehat\beta_{i,\gamma}, \beta_i) h(\widehat\beta_{p+\pi(i), \gamma}),
\end{equation}
where $(\pi(1), \ldots, \pi(p))$ is a permutation of $1, \ldots, p$ drawn uniformly at random. Then, first we know
\[
\E\left[\frac1{p}\sum_{i=1}^p g(\widehat\beta_{i,\gamma}, \beta_i) h(\widehat\beta_{p+\pi(i), \gamma}) \Bigg| \F \right] = \left[ \frac1{p}\sum_{i=1}^p g(\widehat\beta_{i,\gamma}, \beta_i) \right] \left[ \frac1{p}\sum_{i=1}^p h(\widehat\beta_{p+i, \gamma}) \right],
\]
which converges to the constant
\[
\E \left[ g(\eta_{\alpha'\tau'^{2-\gamma}, \gamma}(\Pi + \tau' Z), \Pi) \right] \E \left[h(\tau'\eta_{\alpha', \gamma}(Z)) \right].
\]
Recognizing the boundedness of $\sum g h/p$, which results from the boundedness of the terms of this sum, a consequence of the above implies
\begin{equation}\label{eq:e_var_0}
\Var \left\{\E\left[\frac1{p}\sum_{i=1}^p g(\widehat\beta_{i,\gamma}, \beta_i) h(\widehat\beta_{p+i, \gamma}) \Bigg| \F \right]\right\} \goto 0.
\end{equation}
Moreover, due to the boundedness of $\frac1{p}\sum_{i=1}^p g(\widehat\beta_{i,\gamma}, \beta_i) h(\widehat\beta_{p+\pi(i), \gamma})$, it must hold that
\begin{equation}\label{eq:full_mean}
\begin{aligned}
\E\left[\frac1{p}\sum_{i=1}^p g(\widehat\beta_{i,\gamma}, \beta_i) h(\widehat\beta_{p+i, \gamma})\right] &= \E\left[\frac1{p}\sum_{i=1}^p g(\widehat\beta_{i,\gamma}, \beta_i) h(\widehat\beta_{p+\pi(i), \gamma})\right] \\
&\Longrightarrow \E \left[ g(\eta_{\alpha'\tau'^{2-\gamma}, \gamma}(\Pi + \tau' Z), \Pi) \right] \E \left[h(\tau'\eta_{\alpha', \gamma}(Z)) \right].
\end{aligned}
\end{equation}
Now, we consider the variance and write $\|f\|_{\infty}$ for the supremum of a function $f$. To begin, we invoke Lemma \ref{lm:ab_perm}, from which we get
\[
\begin{aligned}
\Var\left[\frac1{p}\sum_{i=1}^p g(\widehat\beta_{i,\gamma}, \beta_i) h(\widehat\beta_{p+i, \gamma}) \Bigg| \F \right] &= \frac{\sum_{i=1}^p (g(\widehat\beta_{i,\gamma}, \beta_i) - \overbar g)^2 \sum_{i=1}^p (h(\widehat\beta_{p+i,r}) - \overbar h)^2}{p^2(p-1)}\\
&\le \frac{\sum_{i=1}^p 4 \|g\|_{\infty}^2 \sum_{i=1}^p 4 \|h\|_{\infty}^2}{p^2(p-1)}\\
&\le \frac{16 p^2 \|g\|_{\infty}^2 \|h\|_{\infty}^2}{p^2(p-1)}\\
&\le \frac{16 \|g\|_{\infty}^2 \|h\|_{\infty}^2}{p-1}\\
& \goto 0
\end{aligned}
\]
as $p \goto \infty$. Therefore, its boundedness gives
\begin{equation}\label{eq:var_e_0}
\E \left\{\Var\left[\frac1{p}\sum_{i=1}^p g(\widehat\beta_{i,\gamma}, \beta_i) h(\widehat\beta_{p+i, \gamma}) \Bigg| \F \right]\right\} \goto 0.
\end{equation}
Thus, from \eqref{eq:e_var_0} and \eqref{eq:var_e_0} we get
\begin{equation}\label{eq:full_var}
\begin{aligned}
&\Var\left[\frac1{p}\sum_{i=1}^p g(\widehat\beta_{i,\gamma}, \beta_i) h(\widehat\beta_{p+i, \gamma})\right] \\
&= \Var \left\{\E\left[\frac1{p}\sum_{i=1}^p g(\widehat\beta_{i,\gamma}, \beta_i) h(\widehat\beta_{p+i, \gamma}) \Bigg| \F \right]\right\} + \E \left\{\Var\left[\frac1{p}\sum_{i=1}^p g(\widehat\beta_{i,\gamma}, \beta_i) h(\widehat\beta_{p+i, \gamma}) \Bigg| \F \right]\right\}\\
& \goto 0.
\end{aligned}
\end{equation}

Finally, \eqref{eq:full_mean} and \eqref{eq:full_var} together reveal that
\[
\frac1{p}\sum_{i=1}^p g(\widehat\beta_{i,\gamma}, \beta_i) h(\widehat\beta_{p+i, \gamma}) \Longrightarrow \E \left[ g(\eta_{\alpha'\tau'^{2-\gamma}, \gamma}(\Pi + \tau' Z), \Pi) \right] \E \left[h(\tau'\eta_{\alpha', \gamma}(Z)) \right]
\]
as $p \goto \infty$.

\end{proof}


In the remainder of this subsection, we complete the proof of Lemmas~\ref{lm:part_amp}, \ref{lm:p_part}, and \ref{lm:ab_perm}. In the proof of Lemma~\ref{lm:part_amp}, we need the following preparatory lemma.

\begin{lemma}\label{lm:exchangeable}
Let $\{\xi_{p1}, \xi_{p2}, \ldots, \xi_{p m_{p}}\}_{p=1}^{\infty}$ be
a triangular array of bounded random variables such that $\xi_{p1},
\xi_{p2}, \ldots, \xi_{p m_{p}}$ are exchangeable for every $p$ and
$m_p \goto \infty$ as $p \goto \infty$. If for a constant $c$, 
\[
\frac{\xi_{p1} + \cdots + \xi_{p m_{p}}}{m_p} \Longrightarrow c
\]
as $p \goto \infty$,  for an arbitrary (deterministic) sequence $l_p$ satisfying $l_p \le m_p$ and $l_p \goto \infty$, we must have
\[
\frac{\xi_{p1} + \cdots + \xi_{p l_{p}}}{l_p} \Longrightarrow c.
\]

\end{lemma}

\begin{proof}[Proof of Lemma \ref{lm:exchangeable}]
Fix any $\iota > 0$. We will show that
\[\lim_{p\rightarrow\infty}\mathbb{P}\left\{\left|\frac{\xi_{p1} + \dots + \xi_{pl_p}}{l_p} - c\right|>\iota\right\}= 0.\]
For any $p$ let $S_p$ be a random subset of $\{1,\dots,m_p\}$ of cardinality $l_p$, drawn independently of the $\xi_{pi}'s$. Then by exchangeability,
$\sum_{i=1}^{l_p}\xi_{pi}$ is equal in distribution to $\sum_{i\in S_p}\xi_{pi}$.
Therefore we equivalently need to show that
\[\lim_{p\rightarrow\infty}\mathbb{P}\left\{\left|\frac{\sum_{i\in S_p}\xi_{pi}}{l_p} - c\right|>\iota\right\}= 0.\]
We trivially have 
\begin{multline*}\lim_{p\rightarrow\infty}\mathbb{P}\left\{\left|\frac{\sum_{i\in S_p}\xi_{pi}}{l_p} - c\right|>\iota\right\} \\
\leq \lim_{p\rightarrow\infty}\mathbb{P}\left\{\left|\frac{\sum_{i\in S_p}\xi_{pi}}{l_p} - \frac{\xi_{p1} + \dots + \xi_{pm_p}}{m_p}\right|>\iota/2\right\}
+\lim_{p\rightarrow\infty}\mathbb{P}\left\{\left|\frac{\xi_{p1} + \dots + \xi_{pm_p}}{m_p} - c\right|>\iota/2\right\}.\end{multline*}
The assumption $\frac{\xi_{p1}+\dots+\xi_{pm_p}}{m_p} \Longrightarrow c$ implies that
\[\lim_{p\rightarrow\infty}\mathbb{P}\left\{\left|\frac{\xi_{p1} + \dots + \xi_{pm_p}}{m_p} - c\right|>\iota/2\right\}= 0.\]
Next we bound the remaining term. Recall that the $\xi_{pi}$'s are bounded, so we can assume $\xi_{pi}\in[-B,B]$ for some finite $B>0$. We then have
\[\mathrm{Var}\left(\frac{\sum_{i\in S_p}\xi_{pi}}{l_p} 
\  \bigg\vert \ \xi_{p1},\dots,\xi_{pm_p}\right) \leq \frac{4B^2}{l_p},\]
since sampling uniformly with replacement always has variance no larger than sampling uniformly without replacement, and the $\xi_{pi}$'s are bounded. 
Therefore,
\[\mathbb{P}\left\{\left|\frac{\sum_{i\in S_p}\xi_{pi}}{l_p} - \frac{\xi_{p1} + \dots + \xi_{pm_p}}{m_p}\right|>\iota/2 \  \bigg\vert \  \xi_{p1},\dots,\xi_{pm_p}\right\}
\leq \frac{4B^2/l_p}{\iota^2/4}\]
almost surely. Marginalizing,
\[\mathbb{P}\left\{\left|\frac{\sum_{i\in S_p}\xi_{pi}}{l_p} - \frac{\xi_{p1} + \dots + \xi_{pm_p}}{m_p}\right|>\iota/2\right\}
\leq \frac{4B^2/l_p}{\iota^2/4},\]
which tends to zero as $p\rightarrow\infty$ since $\iota$ is fixed and $l_p\rightarrow\infty$. This completes the proof.
\end{proof}




Now, we are ready to prove Lemma \ref{lm:part_amp}.
\begin{proof}[Proof of Lemma \ref{lm:part_amp}]
It suffices to prove the lemma for any bounded Lipschitz continuous functions. To see this sufficiency, assume for the moment that 
\begin{equation}\label{eq:g_assume}
\frac1{p}\sum_{i=1}^p g(\widehat\beta_{p+i, \gamma}) \Longrightarrow \E g(\tau' \eta_{\alpha', \gamma}(Z))
\end{equation}
if $g$ is bounded and Lipschitz continuous. Let $f$ be a continuous function that satisfies $|f(x)| \le M$ for all $x$. We show below that
\begin{equation}\label{eq:f_holds_t}
\frac1{p}\sum_{i=1}^p f(\widehat\beta_{p+i, \gamma}) \Longrightarrow \E f(\tau' \eta_{\alpha', \gamma}(Z)).
\end{equation}

Let $\upsilon > 0$ be an arbitrary small number. As a consequence of Lemma~\ref{lm:boundedbeta} presented in Section~\ref{sec:proof-theorem} below, if $A$ is sufficiently large, then
\begin{equation}\label{eq:a_most_s}
\frac{\#\{1 \le i \le p: |\widehat\beta_{p+i, \gamma}| > A \}}{p} \le \upsilon
\end{equation}
with probability tending to one as $p \goto \infty$. As is clear, one can find a Lipschitz continuous function $g$ defined on a compact set, for example, $[-A, A]$ that satisfies
\begin{equation}\label{eq:f_g_close}
|f(x) - g(x)| \le \upsilon
\end{equation}
for all $-A \le x \le A$. We can extend $g$ to a bounded Lipschitz continuous function defined on $\R$. This can be done, for example, by setting $g(x) = 0$ if $|x| > A + 1$ and let $g$ be linear on $[-A-1, -A]$ and $[A, A+1]$. Hence, \eqref{eq:g_assume} holds for $g$. Let $M'$ be an upper bound of $g$ in the sense that $|g(x)| \le M'$ for all $x$ (we can take $M' = M + \upsilon$). To show \eqref{eq:f_holds_t}, we first write 
\[
\frac1{p}\sum_{i=1}^p f(\widehat\beta_{p+i, \gamma}) = \frac1{p}\sum_{i=1}^p f(\widehat\beta_{p+i, \gamma})\bm{1}_{|\widehat\beta_{p+i, \gamma}| \le A} + \frac1{p}\sum_{i=1}^p f(\widehat\beta_{p+i, \gamma})\bm{1}_{|\widehat\beta_{p+i, \gamma}| > A}
\]
and
\[
\frac1{p}\sum_{i=1}^p g(\widehat\beta_{p+i, \gamma}) = \frac1{p}\sum_{i=1}^p g(\widehat\beta_{p+i, \gamma})\bm{1}_{|\widehat\beta_{p+i, \gamma}| \le A} + \frac1{p}\sum_{i=1}^p g(\widehat\beta_{p+i, \gamma})\bm{1}_{|\widehat\beta_{p+i, \gamma}| > A},
\]
where the indicator function $\bm{1}$ takes the value 1 if the event in the subscript happens and takes the value 0 otherwise. This gives
\[
\begin{aligned}
\left| \frac1{p}\sum_{i=1}^p f(\widehat\beta_{p+i,r}) - \frac1{p}\sum_{i=1}^p g(\widehat\beta_{p+i,r}) \right| \le& \left| \frac1{p}\sum_{i=1}^p f(\widehat\beta_{p+i,r})\bm{1}_{|\widehat\beta_{p+i,r}| \le A} - \frac1{p}\sum_{i=1}^p g(\widehat\beta_{p+i,r})\bm{1}_{|\widehat\beta_{p+i,r}| \le A} \right| \\
&+ \left| \frac1{p}\sum_{i=1}^p f(\widehat\beta_{p+i,r})\bm{1}_{|\widehat\beta_{p+i,r}| > A} \right| + \left| \frac1{p}\sum_{i=1}^p g(\widehat\beta_{p+i,r})\bm{1}_{|\widehat\beta_{p+i,r}| > A} \right|\\
\le& \frac1{p}\sum_{i=1}^p |f(\widehat\beta_{p+i,r}) - g(\widehat\beta_{p+i,r})|\bm{1}_{|\widehat\beta_{p+i,r}| \le A} \\
&+ \frac1{p}\sum_{i=1}^p M \bm{1}_{|\widehat\beta_{p+i,r}| > A} + \frac1{p}\sum_{i=1}^p M' \bm{1}_{|\widehat\beta_{p+i,r}| > A}\\
\le& \frac1{p}\sum_{i=1}^p \upsilon \bm{1}_{|\widehat\beta_{p+i,r}| \le A} +  \frac{(M + M') \#\{1 \le i \le p: |\widehat\beta_{p+i,r}| > A\}}{p}  \\
\le& \upsilon +  (M + M') \upsilon \\
=& (M + M' + 1) \upsilon,
\end{aligned}
\]
where in the second last inequality we use \eqref{eq:f_g_close}, and the last inequality follows from \eqref{eq:a_most_s} and thus holds with probability tending to one. Similarly, we can show that the difference between $\E g(\tau'\eta_{\alpha', \gamma}(Z))$ and $\E f(\tau'\eta_{\alpha', \gamma}(Z))$ can be made arbitrarily small if $\upsilon$ is small enough. Taking $\upsilon \goto 0$, therefore, we see that \eqref{eq:g_assume} implies \eqref{eq:f_holds_t}.

To conclude the proof of this lemma, therefore, it is sufficient to prove \eqref{eq:g_assume} for any bounded Lipschitz continuous function $g$. For convenience, we write $f$ in place of $g$ and assume that $f$ is bounded by $M$ in magnitude and is $L$-Lipschitz continuous. Consider the function
\[
f_a(x, y) = f(x) \max\{0, 1 - |y|/a\}
\]
for $a > 0$. Our first step is to verify that this function is Lipschitz continuous and is, therefore, pseudo-Lipschitz continuous (see the definition in \cite{bayati2012}). Writing $x_+$ for $\max\{0, x\}$, we note that
\[
\begin{aligned}
|f_a(x, y) - f_a(x', y')| &= \left| f(x) (1 - |y|/a)_+ - f(x') (1 - |y'|/a)_+ \right|\\
                                &= \left| f(x) (1 - |y|/a)_+ - f(x) (1 - |y'|/a)_+ + f(x) (1 - |y'|/a)_+ - f(x') (1 - |y'|/a)_+ \right|\\
                                &\le \left| f(x) (1 - |y|/a)_+ - f(x) (1 - |y'|/a)_+ \right| + \left| f(x) (1 - |y'|/a)_+ - f(x') (1 - |y'|/a)_+ \right|\\
                                &\le M |y - y'|/a + L (1 - |y'|/a)_+ |x - x'| \\
                                &\le (M/a + L)\|(x, y) - (x', y')\|_2.
\end{aligned}
\]
This proves that $f_a$ is Lipschitz continuous. Using Theorem 2.1 of \cite{weng2018overcoming}, therefore, we get
\begin{equation}\label{eq:a_small_zero}
\frac1{2p} \sum_{i=1}^{2p} f_a(\widehat\beta_{i,\gamma}, \beta_i) \Longrightarrow \E f_a\left( \eta_{\alpha'\tau'^{2-\gamma}, \gamma}(\widetilde\Pi + \tau' Z), \widetilde\Pi \right) 
\end{equation}
for any fixed $a > 0$, where the random variable $\widetilde\Pi = \Pi$ with probability $\frac12$ and otherwise $\widetilde\Pi = 0$. 
Note that Theorem 2.1 in its present form considers a bridge estimator of order $1 \le \gamma \le 2$ and can be extended to any $\gamma > 2$ (personal communication).


Now we will take $a\rightarrow0$ on the right-hand side of \eqref{eq:a_small_zero}. Recognizing that 
\[
f_a\left( \eta_{\alpha'\tau'^{2-\gamma}, \gamma}(\widetilde\Pi + \tau' Z), \widetilde\Pi \right) \goto f\left( \eta_{\alpha'\tau'^{2-\gamma}, \gamma}(\widetilde\Pi + \tau' Z)\right)
\] 
if $\widetilde\Pi = 0$ and otherwise $f_a\left( \eta_{\alpha'\tau'^{2-\gamma}, \gamma}(\widetilde\Pi + \tau' Z), \widetilde\Pi \right) \goto 0$ as $a \goto 0$, the boundedness of $f_a$ allows us to use Lebesgue's dominated convergence theorem to obtain
\begin{equation}\label{eq:a_limit}
\begin{aligned}
\lim_{a \goto 0+} \E f_a\left( \eta_{\alpha'\tau'^{2-\gamma}, \gamma}(\widetilde\Pi + \tau' Z), \widetilde\Pi \right) &= \frac12 \E f\left( \eta_{\alpha'\tau'^{2-\gamma}, \gamma}(\tau' Z)\right) + \frac{1 - \epsilon}{2} \E f\left( \tau'\eta_{\alpha',r}(Z)\right)\\
&= \frac{2-\epsilon}2 \E f\left( \tau'\eta_{\alpha', \gamma}(Z)\right).
\end{aligned}
\end{equation}
Turning to the left-hand side of \eqref{eq:a_small_zero}, we use the fact that for any $c_1 > 0$, one can find $c_2 > 0$ such that
\begin{equation}\label{eq:sum_non_small}
\left| \frac1{2p} \sum_{i: \beta_i \ne 0} f_a(\widehat\beta_{i,\gamma}, \beta_i) \right| \le c_1
\end{equation}
with probability approaching one for each $a < c_2$. To see this, note that
\[
\begin{aligned}
\left| \frac1{2p} \sum_{i: \beta_i \ne 0} f_a(\widehat\beta_{i,\gamma}, \beta_i) \right| &\le \frac1{2p} \sum_{i: \beta_i \ne 0} \left|f_a(\widehat\beta_{i,\gamma}, \beta_i) \right|\\
&\le \frac1{2p} \sum_{i: \beta_i \ne 0} M (1 - |\beta_i|/a)_+,
\end{aligned}
\]
of which the expectation satisfies
\[
\E \left[ \frac1{2p} \sum_{i: \beta_i \ne 0} M (1 - |\beta_i|/a)_+ \right] = \frac{\epsilon}2 \E \left[ M (1 - |\Pi^*|/a)_+ \right] \le \frac{M \epsilon}{2} \P(|\Pi^*| < a),
\]
since $\Pi^*$ places no mass at zero, by definition. This inequality in conjunction with the Markov inequality reveals that \eqref{eq:sum_non_small} holds if $a$ is sufficiently small.


Writing
\[
\frac1{2p} \sum_{i=1}^{2p} f_a(\widehat\beta_{i,\gamma}, \beta_i) = \frac1{2p} \sum_{i: \beta_i \ne 0} f_a(\widehat\beta_{i,\gamma}, \beta_i) + \frac1{2p} \sum_{i: \beta_i = 0} f(\widehat\beta_{i,\gamma})
\]
and taking $a \goto 0$, we get
\begin{equation}\nonumber
\frac1{2p} \sum_{1 \le i \le 2p: \beta_i = 0} f(\widehat\beta_{i,\gamma}) \Longrightarrow \frac{2-\epsilon}2 \E f\left( \tau'\eta_{\alpha', \gamma}(Z)\right)
\end{equation}
from \eqref{eq:a_limit}, \eqref{eq:a_small_zero}, and \eqref{eq:sum_non_small}. This is equivalent to
\begin{equation}\label{eq:beta_div_nom}
\frac{\sum_{1 \le i \le 2p: \beta_i = 0} f(\widehat\beta_{i,\gamma})}{\#\{1 \le
  i \le 2p: \beta_i = 0\}} \Longrightarrow \E f\left( \tau'\eta_{\alpha', \gamma}(Z)\right),
\end{equation}
which makes use of the fact that
\begin{equation}\label{eq:mp_concentrate}
\frac{\#\{1 \le i \le 2p: \beta_i = 0\}}{2p} \Longrightarrow \frac{2 -
\epsilon}{2}.
\end{equation}

To conclude the proof of this lemma, we apply Lemma \ref{lm:exchangeable} to \eqref{eq:beta_div_nom}. This is done by letting $m_p = \#\{1 \le
  i \le 2p: \beta_i = 0\}$ and $\{\xi_{p1}, \xi_{p2}, \ldots, \xi_{p
    m_{p}}\} = \{f(\widehat\beta_{i,\gamma}): \beta_i = 0, 1 \le i \le 2p\}$
  and $l_p = p$ and $\{\xi_{p1}, \xi_{p2}, \ldots, \xi_{p l_{p}}\} =
  \{f(\widehat\beta_{i,\gamma}): \beta_i = 0, p+1 \le i \le 2p\}$. For completeness, we remark that the randomness of $m_{p}$ does not affect the validity of Lemma~\ref{lm:exchangeable} due to \eqref{eq:mp_concentrate}. Thus, we get
\[
\frac1{p} \sum_{i=p+1}^{2p} f(\widehat\beta_{i,\gamma}) = \frac1{l_p} \sum_{i=p+1}^{2p} f(\widehat\beta_{i,\gamma}) \Longrightarrow \E f\left( \tau'\eta_{\alpha', \gamma}(Z)\right).
\]
This completes the proof.

\end{proof}

\begin{proof}[Proof of Lemma \ref{lm:p_part}]
As with Lemma \ref{lm:part_amp}, it is sufficient to prove the present lemma for any bounded Lipschitz continuous functions. By Theorem 1.5 of \cite{bayati2012}, we get
\begin{equation}\label{eq:bayati_apply}
\begin{aligned}
\frac1{2p}\sum_{i=1}^{2p} f(\widehat\beta_{i,\gamma}, \beta_i) = \frac1{2p}\sum_{i=1}^p f(\widehat\beta_{i,\gamma}, \beta_i) + \frac1{2p}\sum_{i=p+1}^{2p} f(\widehat\beta_{i,\gamma}, 0) \Longrightarrow \E f(\eta_{\alpha'\tau'^{2-\gamma}, \gamma}(\widetilde\Pi+\tau' Z), \widetilde\Pi).
\end{aligned}
\end{equation}
Note that the right-hand side can be written as
\begin{equation}\label{eq:exp_2_t}
\E f(\eta_{\alpha'\tau'^{2-\gamma}, \gamma}(\widetilde\Pi+\tau' Z), \widetilde\Pi) = \frac12 \E f(\eta_{\alpha'\tau'^{2-\gamma}, \gamma}(\Pi+\tau' Z), \Pi) + \frac12 \E f(\eta_{\alpha'\tau'^{2-\gamma}, \gamma}(\tau' Z), 0).
\end{equation}

On the other hand, from Lemma \ref{lm:part_amp} we know
\begin{equation}\label{eq:early_lm_res}
\frac1{p}\sum_{i=p+1}^{2p} f(\widehat\beta_{i,\gamma}, 0) \Longrightarrow \E f(\tau'\eta_{\alpha', \gamma}(Z), 0).
\end{equation}
Plugging \eqref{eq:early_lm_res} into \eqref{eq:bayati_apply} and recognizing \eqref{eq:exp_2_t}, we get
\[
\frac1{p}\sum_{i=1}^p f(\widehat\beta_{i,\gamma}, \beta_i) \Longrightarrow \E f(\eta_{\alpha'\tau'^{2-\gamma}, \gamma}(\Pi+\tau' Z), \Pi).
\]
This completes the proof.

\end{proof}

\begin{proof}[Proof of Lemma \ref{lm:ab_perm}]
We have
\[
\begin{aligned}
\Var(A_1 B_{\pi(1)} + \cdots + A_p B_{\pi(p)}) = \sum_{i=1}^p \Var(A_i B_{\pi(i)}) + 2 \sum_{i < j} \Cov(A_i B_{\pi(i)}, A_j B_{\pi(j)}).
\end{aligned}
\]
First, we get
\[
\Var(A_i B_{\pi(i)}) = A_i^2\Var(B_{\pi(i)}) = \frac{A_i^2 \sum_{l=1}^p (B_l - \overbar B)^2}{p},
\]
where $\overbar B = (B_1 + \cdots + B_p)/p$, and
\[
\begin{aligned}
\Cov(A_i B_{\pi(i)}, A_j B_{\pi(j)}) &= A_i A_j \Cov(B_{\pi(i)}, B_{\pi(j)})\\
&= A_i A_j \left( \E B_{\pi(i)} B_{\pi(j)} - \E B_{\pi(i)} \E B_{\pi(j)} \right)\\
&= A_i A_j \left( \frac{\sum_{l \ne m} B_l B_m}{p(p-1)} - \frac{(B_1 + \cdots + B_p)^2}{p^2}\right)\\
&= -A_i A_j \frac{\sum_{l=1}^p (B_l - \overbar B)^2}{p(p-1)}.
\end{aligned}
\]
Thus, we get
\[
\begin{aligned}
&\sum_{i=1}^p \Var(A_i B_{\pi(i)}) + 2 \sum_{i < j} \Cov(A_i B_{\pi(i)}, A_j B_{\pi(j)}) \\
&= \sum_{i=1}^p \frac{A_i^2 \sum_{l=1}^p (B_l - \overbar B)^2}{p} - 2\sum_{i < j} A_i A_j \frac{\sum_{l=1}^p (B_l - \overbar B)^2}{p(p-1)}\\
&= \left[\sum_{l=1}^p (B_l - \overbar B)^2 \right] \left[ \sum_{i=1}^p \frac{A_i^2}{p} - 2\sum_{i < j} \frac{A_i A_j}{p(p-1)} \right]\\
&= \left[\sum_{l=1}^p (B_l - \overbar B)^2 \right] \left[ \frac{\sum_{l=1}^p (A_l - \overbar A)^2}{p-1} \right]\\
&= \frac{\left[\sum_{l=1}^p (B_l - \overbar B)^2 \right] \left[\sum_{l=1}^p (A_l - \overbar A)^2 \right]}{p-1}.
\end{aligned}
\]

\end{proof}

\subsection{Proofs of Theorem \ref{thm:contrast_gen} and Proposition~\ref{thm:contrast}}
\label{sec:proof-theorem}

We first prove Theorem \ref{thm:contrast_gen} with a fixed $\lambda$, followed by a discussion showing that the theorem holds uniformly over $\lambda$ in a compact set for the Lasso, thereby proving Proposition~\ref{thm:contrast}. In addition to Lemma~\ref{lm:factored}, the proof relies on Lemmas~\ref{lm:sw_app} and \ref{lm:boundedbeta}, which we state below.

Let $C(\Omega, \R)$ denote the class of all real-valued continuous functions defined on a compact Hausdorff space $\Omega$.

\begin{lemma}\label{lm:sw_app}
Let $\Omega_1$ and $\Omega_2$ be two compact Hausdorff spaces and $f :
\Omega_1 \times \Omega_2 \rightarrow \R$ be a continuous function,
then for every $\upsilon > 0$ there exist a positive integer $m$ and continuous
functions $g_1, \ldots, g_m$ on $\Omega_1$ and continuous functions $h_1, \ldots, h_m$ on $\Omega_2$ such that 
\[
\sup_{(x_1, x_2) \in \Omega_1 \times \Omega_2} \left| f(x_1, x_2) -
  \sum_{i=1}^m  g_i(x_1)  h_i(x_2) \right| \le \upsilon.
\]
\end{lemma}

Lemma~\ref{lm:sw_app} serves as an approximation tool for our proof. For information, this lemma follows from the Stone--Weierstrass theorem (see Corollary 11.6 in \cite{carothers1998short}).

\begin{lemma}\label{lm:boundedbeta}
\[
\lim_{A \goto \infty} \limsup_{p \goto \infty} \E \left[ \frac{\#\{1 \le i \le p: \max(|\widehat\beta_{i,\gamma}|, |\beta_i|, |\widehat\beta_{p+i,r}|) > A\}}{p} \right]  = 0.
\]
\end{lemma}

\begin{proof}[Proof of Lemma \ref{lm:boundedbeta}]
Note that we have
\[
\begin{aligned}
&\#\{1 \le i \le p: \max(|\widehat\beta_{i,\gamma}|, |\beta_i|, |\widehat\beta_{p+i,r}|) > A\} \\
&\le \#\{1 \le i \le p: |\widehat\beta_{i,\gamma}| > A\} + \#\{1 \le i \le p: |\beta_i| > A\} + \#\{p+1 \le i \le 2p: |\widehat\beta_{i,\gamma}| > A\}.
\end{aligned}
\]
It follows from \cite{weng2018overcoming} that
\[
\frac{\#\{1 \le i \le p: |\widehat\beta_{i,\gamma}| > A\}}{p} \Longrightarrow \P(|\eta_{\alpha'\tau'^{2-\gamma}, \gamma}(\Pi + \tau' Z)| > A),
\]
which tends to $0$ as $A \goto \infty$. Second, 
\[
\frac{\#\{1 \le i \le p: |\beta_i| > A\}}{p} \Longrightarrow \P(|\Pi| > A),
\]
and third, we obtain
\[
\frac{\#\{p+1 \le i \le 2p: |\widehat\beta_{i,\gamma}| > A\}}{p} \Longrightarrow \P(|\tau'\eta_{\alpha', \gamma}(Z)| > A).
\]
Last, note that these fractions are all bounded, so Lebesgue's dominated convergence theorem can be applied here.

\end{proof}

Now we turn to the proof of Theorem \ref{thm:contrast_gen}.
\begin{proof}[Proof of Theorem \ref{thm:contrast_gen}]
Denote by $M$ an upper bound of $f$ in absolute value and let $R > 0$ be a number that will later tend to infinity. It is easy to see that we can construct a continuous function $\tilde f$ defined on $\R^3$ such that (1) $f(x) \equiv \tilde f(x)$ on $B_R \equiv \{x \in \R^3: \|x\|_2 \le R\}$, (2) $|\tilde f(x)| \le M$ for all $x$, and (3) $\lim_{\|x\| \goto \infty} \tilde f(x)$ exists. This can be done, for example, by letting
\[
\tilde f(x) =
\begin{cases}
f(x), \quad \text{if } \|x\|_2 \le R\\
f\left( \frac{Rx}{\|x\|_2} \right) \mathrm{e}^{-\|x\|_2+ R}, \quad \text{otherwise.}
\end{cases}
\]
From the three properties of $\tilde f$, it is easy to see that this is a continuous function on the product of two compact Hausdorff spaces, $\R^2 \cup \{\infty\}$ and $\R \cup \{\infty\}$. From Lemma \ref{lm:sw_app}, therefore, we know that there exist continuous functions $g_1, \ldots, g_m$ on $\R^2 \cup \{\infty\}$ and $h_1, \ldots, h_m$ on $\R \cup \{\infty\}$ such that
\begin{equation}\label{eq:close_stone}
\sup \left|  \tilde f(x_1, x_2, x_3) -
  \sum_{l=1}^m  g_l(x_1, x_2)  h_l(x_3) \right| \le \upsilon
\end{equation}
for any small constant $\upsilon > 0$. 

Since $g_l$ and $h_l$ are continuous on the compactification of their domains for each $l$, the two functions must be continuous and bounded on $\R^2$ and $\R$, respectively. Thus, we get
\begin{equation}\nonumber
\frac1{p}\sum_{i=1}^p   g_l(\widehat\beta_{i,\gamma}, \beta_i)  
h_l(\widehat\beta_{p+i, \gamma}) \dgoto \E \left[  g_l(\eta_{\alpha'\tau'^{2-\gamma}, \gamma}(\Pi + \tau' Z), \Pi) \right] \E \left[   h_l(\tau' \eta_{\alpha', \gamma}(Z')) \right]
\end{equation}
by Lemma \ref{lm:factored}, where $Z$ and $Z'$ are i.i.d.~standard normal random variables. This yields
\begin{equation}\label{eq:g_conv_my}
\begin{aligned}
\frac1{p}\sum_{i=1}^p \sum_{l=1}^m   g_l(\widehat\beta_{i,\gamma},
\beta_i)   h_l(\widehat\beta_{p+i,r}) & \dgoto \sum_{l=1}^m \E \left[   g_l(\eta_{\alpha'\tau'^{2-\gamma}, \gamma}(\Pi + \tau' Z), \Pi) \right] \E \left[   h_l(\tau'\eta_{\alpha', \gamma}(Z')) \right]\\
& =  \E \left[ \sum_{l=1}^m   g_l(\eta_{\alpha'\tau'^{2-\gamma}, \gamma}(\Pi + \tau' Z), \Pi)   h_l(\tau' \eta_{\alpha', \gamma}(Z')) \right].
\end{aligned}
\end{equation}
Taken together, \eqref{eq:close_stone} and \eqref{eq:g_conv_my} give
\begin{equation}\label{eq:f_tilde_bb}
\begin{aligned}
&\P\left( \left| \frac1{p}\sum_{i=1}^p \tilde f(\widehat\beta_{i,\gamma}, \beta_i, \widehat\beta_{p+i, \gamma}) - \E \tilde f(\eta_{\alpha'\tau'^{2-\gamma}, \gamma}(\Pi + \tau' Z), \Pi, \tau'\eta_{\alpha', \gamma}(Z')) \right| < 3\upsilon\right)\\
&\ge \P\left( \left| \frac1{p}\sum_{i=1}^p \sum_{l=1}^m   g_l(\widehat\beta_{i,\gamma},
\beta_i)   h_l(\widehat\beta_{p+i, \gamma}) - \E \left[ \sum_{l=1}^m   g_l(\eta_{\alpha'\tau'^{2-\gamma}, \gamma}(\Pi + \tau' Z), \Pi)   h_l(\tau'\eta_{\alpha', \gamma}(Z')) \right] \right| < \upsilon\right)\\
& \rightarrow 1
\end{aligned}
\end{equation}
as $p \goto \infty$.

Next, we consider
\begin{equation}\label{eq:diff_1_exp}
\begin{aligned}
\frac1{p}\sum_{i=1}^p f(\widehat\beta_{i,\gamma}, \beta_i, \widehat\beta_{p+i, \gamma}) - \frac1{p}\sum_{i=1}^p \tilde f(\widehat\beta_{i,\gamma}, \beta_i, \widehat\beta_{p+i, \gamma})
\end{aligned}
\end{equation}
and
\begin{equation}\label{eq:diff_2_exp}
\begin{aligned}
\E f(\eta_{\alpha'\tau'^{2-\gamma}, \gamma}(\Pi + \tau' Z), \Pi, \tau'\eta_{\alpha', \gamma}(Z')) - \E \tilde f(\eta_{\alpha'\tau'^{2-\gamma}, \gamma}(\Pi + \tau' Z), \Pi, \tau'\eta_{\alpha', \gamma}(Z')).
\end{aligned}
\end{equation}
Our aim is to show that both displays are small. For the first display, note that
\[
\left| f(\widehat\beta_{i,\gamma}, \beta_i,\widehat\beta_{p+i,r}) - \tilde f(\widehat\beta_{i,\gamma}, \beta_i,\widehat\beta_{p+i,r}) \right| 
 \le 2M \bm{1}_{\|(\widehat\beta_{i,\gamma}, \beta_i,\widehat\beta_{p+i,r})\|_2 > R}
\le 2M \bm{1}_{\max(|\widehat\beta_{i,\gamma}|, |\beta_i|,|\widehat\beta_{p+i,r}|) > R/\sqrt{3}}.
\]
Taking $A = R/\sqrt{3}$, we obtain
\begin{equation}\label{eq:a_r_3}
\left| \frac1{p}\sum_{i=1}^p \left[ f(\widehat\beta_{i,\gamma}, \beta_i,
  \widehat\beta_{p+i,r}) - \tilde f(\widehat\beta_{i,\gamma}, \beta_i,
  \widehat\beta_{p+i,r}) \right]\right| \le \frac{2M \#\{1 \le i \le p: \max(|\widehat\beta_{i,\gamma}|, |\beta_i|,
  |\widehat\beta_{p+i,r}|) > A\}}{p}.
\end{equation}
Likewise, we show below that \eqref{eq:diff_2_exp} can be made arbitrarily small in absolute value. To this end, note that
\begin{equation}\label{eq:r_3_two}
\begin{aligned}
&\left| \E \left[ f(\eta_{\alpha'\tau'^{2-\gamma}, \gamma}(\Pi + \tau' Z), \Pi,
    \tau'\eta_{\alpha', \gamma}(Z')) - \tilde f(\eta_{\alpha'\tau'^{2-\gamma}, \gamma}(\Pi + \tau' Z), \Pi,
    \tau'\eta_{\alpha', \gamma}(Z'))\right] \right| \\
&\le \E \left| f(\eta_{\alpha'\tau'^{2-\gamma}, \gamma}(\Pi + \tau' Z), \Pi,
    \tau'\eta_{\alpha', \gamma}(Z')) - \tilde f(\eta_{\alpha'\tau'^{2-\gamma}, \gamma}(\Pi + \tau' Z), \Pi,
    \tau'\eta_{\alpha', \gamma}(Z'))\right| \\
&\le 2M \P\left( \max(|\eta_{\alpha'\tau'^{2-\gamma}, \gamma}(\Pi + \tau' Z)|, |\Pi|,
    |\tau' \eta_{\alpha', \gamma}(Z')|) > A \right).
\end{aligned}
\end{equation}


Finally, from \eqref{eq:f_tilde_bb}, \eqref{eq:a_r_3}, and \eqref{eq:r_3_two} it follows that the event 
\[
\begin{aligned}
 &\left| \frac1{p}\sum_{i=1}^p f(\widehat\beta_{i,\gamma}, \beta_i, \widehat\beta_{p+i,r}) - \E f(\eta_{\alpha'\tau'^{2-\gamma}, \gamma}(\Pi + \tau' Z), \Pi, \tau'\eta_{\alpha', \gamma}(Z')) \right| \\
&< 3\upsilon + \frac{2M \#\{1 \le i \le p: \max(|\widehat\beta_{i,\gamma}|, |\beta_i|,
  |\widehat\beta_{p+i,r}|) > A\}}{p} \\
&+ 2M \P\left( \max(|\eta_{\alpha'\tau'^{2-\gamma}, \gamma}(\Pi + \tau' Z)|, |\Pi|,
    |\tau'\eta_{\alpha', \gamma}(Z')|) > A \right)
\end{aligned}
\]
happens with probability tending to one as $p \goto \infty$. Taking $A \equiv R/\sqrt{3} \goto \infty$ followed by letting $\upsilon \goto 0$, Lemma~\ref{lm:boundedbeta} shows that
\[
3\upsilon + \frac{2M \#\{1 \le i \le p: \max(|\widehat\beta_{i,\gamma}|, |\beta_i|,
  |\widehat\beta_{p+i,r}|) > A\}}{p} + 2M \P\left( \max(|\eta_{\alpha'\tau'^{2-\gamma}, \gamma}(\Pi + \tau' Z)|, |\Pi|,
    |\tau'\eta_{\alpha', \gamma}(Z')|) > A \right)
\] 
can be made arbitrarily small. This reveals that
\[
\frac1{p}\sum_{i=1}^p f(\widehat\beta_{i,\gamma}, \beta_i, \widehat\beta_{p+i,r}) \dgoto \E \left[ f(\eta_{\alpha'\tau'^{2-\gamma}, \gamma}(\Pi + \tau' Z), \Pi, \tau'\eta_{\alpha', \gamma}(Z')) \right],
\]
thereby completing the proof.

\end{proof}



Next, we discuss how Corollary~\ref{cor:tpp-fdp-inf} and Corollary~\ref{cor:tinf} can be derived from Theorem \ref{thm:contrast_gen}.

\begin{proof}[Proof of Corollaries~\ref{cor:tpp-fdp-inf} and \ref{cor:tinf}]
Define 
\[
f_a(x, y, z) = \max\left\{0, 1 + \min\left\{ 0, \frac{|x| - |z| - t}{a} \right\} \right\} \cdot \max\{0, 1 - |y|/a\}
\]
for $a > 0$. As is clear, $f_a$ is bounded and continuous. Therefore, by Theorem~\ref{thm:contrast_gen} we get
\[
\frac1{p}\sum_{i=1}^p f_a(\widehat\beta_{i,\gamma}, \beta_i, \widehat\beta_{p+i, \gamma}) \dgoto \E f_a(\eta_{\alpha'\tau'^{2-\gamma}, \gamma}(\Pi + \tau' Z), \Pi, \tau'\eta_{\alpha', \gamma}(Z'))
\]
as $p \goto \infty$. On the one hand, by the same argument for \eqref{eq:a_small_zero}, we obtain
\begin{equation}\label{eq:useful_cor}
\begin{aligned}
&\lim_{a \goto 0} \E f_a(\eta_{\alpha'\tau'^{2-\gamma}, \gamma}(\Pi + \tau' Z), \Pi, \tau'\eta_{\alpha', \gamma}(Z')) \\
&= \P(\beta_i = 0) \P ( |\tau' \eta_{\alpha', \gamma}(Z)| - |\tau' \eta_{\alpha', \gamma}(Z')| \ge t ) \\
&= (1- \epsilon) \P ( |\tau' \eta_{\alpha', \gamma}(Z)| - |\tau' \eta_{\alpha', \gamma}(Z')| \ge t ).
\end{aligned}
\end{equation}

On the other hand, note that the number of false discoveries at threshold value $t$ is
\[
\sum_{i=1}^p \bm{1}_{\beta_i = 0,  |\widehat\beta_{i,\gamma}| - |\widehat\beta_{p+i, \gamma}| \ge t}.
\]
It is easy to see that
\[
f_a(\widehat\beta_{i,\gamma}, \beta_i, \widehat\beta_{p+i, \gamma}) - \bm{1}_{0 < |\beta_i| < a} - \bm{1}_{t - a< |\widehat\beta_{i,\gamma}| - |\widehat\beta_{p+i, \gamma}| < t} \le \bm{1}_{\beta_i = 0,  |\widehat\beta_{i,\gamma}| - |\widehat\beta_{p+i, \gamma}| \ge t} \le f_a(\widehat\beta_{i,\gamma}, \beta_i, \widehat\beta_{p+i, \gamma}).
\]
Thus, we have
\[
\left| \frac1p \sum_{i=1}^p \bm{1}_{\beta_i = 0,  |\widehat\beta_{i,\gamma}| - |\widehat\beta_{p+i, \gamma}| \ge t} - \frac1{p}\sum_{i=1}^p f_a(\widehat\beta_{i,\gamma}, \beta_i, \widehat\beta_{p+i, \gamma}) \right| \le \frac1p \sum_{i=1}^p \left[ \bm{1}_{0 < |\beta_i| < a} + \bm{1}_{t - a< |\widehat\beta_{i,\gamma}| - |\widehat\beta_{p+i, \gamma}| < t} \right].
\]
Using Theorem~\ref{thm:contrast_gen} with appropriate bounded continuous functions, we can show that
\[
\frac{\sum_{i=1}^p \left[ \bm{1}_{0 < |\beta_i| < a} + \bm{1}_{t - a< |\widehat\beta_{i,\gamma}| - |\widehat\beta_{p+i, \gamma}| < t} \right] }{p} \le c_a
\]
in probability for a constant $c_a > 0$ satisfying $c_a \goto 0$ as $a \goto 0$. Together with \eqref{eq:useful_cor}, this gives
\[
\frac1p \sum_{i=1}^p \bm{1}_{\beta_i = 0,  |\widehat\beta_{i,\gamma}| - |\widehat\beta_{p+i, \gamma}| \ge t} \dgoto (1- \epsilon) \P ( |\tau' \eta_{\alpha', \gamma}(Z)| - |\tau' \eta_{\alpha', \gamma}(Z')| \ge t )
\]
as $p \goto \infty$. Similarly, one can show that
\[
\frac1p \sum_{i=1}^p \bm{1}_{|\widehat\beta_{i,\gamma}| - |\widehat\beta_{p+i, \gamma}| \ge t} \dgoto \P ( |\eta_{\alpha'\tau'^{2-\gamma}, \gamma}(\Pi + \tau' Z)| - |\tau' \eta_{\alpha', \gamma}(Z')| \ge t ).
\]
This proves the first identity in Corollary~\ref{cor:tpp-fdp-inf}. The second identify of Corollary~\ref{cor:tpp-fdp-inf} and Corollary~\ref{cor:tinf} can be proved similarly.

\end{proof}

The remaining part of this subsection is devoted to showing that Theorem~\ref{thm:contrast_gen} holds uniformly over all $\lambda$ in a compact interval of $(0, \infty)$ when $\gamma = 1$. As with the proof of Lemma \ref{lm:part_amp}, we can assume that $f$ is bounded and $L$-Lipschitz continuous. The uniformity extension is accomplished largely by using Lemma B.2 from
\cite{su2017false} (see also \cite{wang2020price}).
\begin{lemma}[Lemma B.2 in \cite{su2017false}]\label{lm:amp_uni}
Fix $0 < \lambda_{\min} < \lambda_{\max}$. Then, there exists a constant
$c$ such that for any $[\lambda^- , \lambda^+] \subset
[\lambda_{\min}, \lambda_{\max}]$, the Lasso estimates satisfy
\[
\sup_{\lambda^- \le \lambda \le \lambda^+} \left\|
  \widehat\beta(\lambda) - \widehat\beta(\lambda^-)\right\|_2 \le c
\sqrt{(\lambda^+ - \lambda^-)p}
\]
with probability tending to one.
\end{lemma}

\begin{proof}[Proof of Proposition~\ref{thm:contrast}]

To begin to establish the uniformity in $\lambda$, let $\lambda_{\min} = \lambda_0 < \lambda_1 < \cdots <
\lambda_m =  \lambda_{\max}$ be equally spaced points and set
  $\Delta \equiv \lambda_{l+1} - \lambda_l = (\lambda_{\max} -
  \lambda_{\min})/m$; We will later take $m \rightarrow \infty$. Write
\[
\ef(\lambda) = \E f(\eta_{\alpha'\tau'}(\Pi + \tau' Z), \Pi, \eta_{\alpha'\tau'}(\tau' Z')).
\]

It follows from Theorem~\ref{thm:contrast_gen} that
\begin{equation}\label{eq:m_union}
\max_{0 \le l \le m} \left| \frac1{p}\sum_{i=1}^p
  f(\widehat\beta_i(\lambda_l), \beta_i, \widehat\beta_{p+i}(\lambda_l)) -
\ef(\lambda_l) \right| \Longrightarrow 0.
\end{equation}
Now, according to Corollary 1.7 from \cite{bayati2012}, both $\alpha', \tau'$ are continuous in
$\lambda$ and, therefore, $\ef(\lambda)$ is also continuous on
$[\lambda_{\min}, \lambda_{\max}]$ . For any constant
$\omega > 0$, therefore, the uniform continuity of $\ef$ ensures that
\begin{equation}\label{eq:rho_contin}
|\ef(\lambda) - \ef(\lambda')| \le \omega
\end{equation}
holds for all
$\lambda_{\min} \le \lambda, \lambda' \le \lambda_{\max}$ satisfying
$|\lambda - \lambda'| \le \Delta$ if $m$ is sufficiently large. Now we
consider
\[
\begin{aligned}
&\left| \frac1{p}\sum_{i=1}^p f(\widehat\beta_i(\lambda), \beta_i,
\widehat\beta_{p+i}(\lambda)) - \frac1{p}\sum_{i=1}^p
  f(\widehat\beta_i(\lambda'), \beta_i, \widehat\beta_{p+i}(\lambda')) \right|\\
&\le \frac1{p}\sum_{i=1}^p \left|  f(\widehat\beta_i(\lambda), \beta_i,
\widehat\beta_{p+i}(\lambda)) - f(\widehat\beta_i(\lambda'), \beta_i, \widehat\beta_{p+i}(\lambda')) \right|\\
&\le \frac1{p}\sum_{i=1}^p L \sqrt{(\widehat\beta_i(\lambda) -
  \widehat\beta_i(\lambda'))^2 + (\beta_i - \beta_i)^2 +
  (\widehat\beta_{p+i}(\lambda) - \widehat\beta_{p+i}(\lambda'))^2}\\
&\le \frac1{p}\sum_{i=1}^p \left( L \left|\widehat\beta_i(\lambda) -
  \widehat\beta_i(\lambda') \right| + L
\left|\widehat\beta_{p+i}(\lambda) - \widehat\beta_{p+i}(\lambda')
\right| \right)\\
&= \frac{L}p \left\| \widehat\beta(\lambda) - \widehat\beta(\lambda')\right\|_1\\
&\le \frac{L}p \sqrt{2p}\left\| \widehat\beta(\lambda) -
  \widehat\beta(\lambda')\right\|_2\\
&= \frac{\sqrt{2} L}{\sqrt{p}} \left\| \widehat\beta(\lambda) - \widehat\beta(\lambda')\right\|_2.
\end{aligned}
\]
Taking $\lambda' = \lambda_l$ for some $l = 0, 1, \ldots, m - 1$ and
$\lambda_l < \lambda \le \lambda_{l+1}$, Lemma \ref{lm:amp_uni} ensures that
\[
\begin{aligned}
&\sup_{\lambda_l \le \lambda \le \lambda_{l+1}}\left| \frac1{p}\sum_{i=1}^p f(\widehat\beta_i(\lambda), \beta_i,
\widehat\beta_{p+i}(\lambda)) - \frac1{p}\sum_{i=1}^p
f(\widehat\beta_i(\lambda_l), \beta_i, \widehat\beta_{p+i}(\lambda_l))
\right| \\
&\le \sup_{\lambda_l \le \lambda \le \lambda_{l+1}}\frac{\sqrt{2} L}{\sqrt{p}} \left\| \widehat\beta(\lambda) - \widehat\beta(\lambda_l)\right\|_2\\
&\le \sup_{\lambda_l \le \lambda \le \lambda_{l+1}}\frac{\sqrt{2}
  L}{\sqrt{p}} c \sqrt{(\lambda - \lambda_l)p}\\
& = \sqrt{2} Lc \sqrt{\frac{\lambda_{\max} - \lambda_{\min}}{m}}\\
& = O(1/\sqrt{m})
\end{aligned}
\]
with probability tending to one. Taking a union bound, we get
\begin{equation}\label{eq:gap_small}
\max_{0 \le l \le m}\sup_{\lambda_l \le \lambda \le \lambda_{l+1}}\left| \frac1{p}\sum_{i=1}^p f(\widehat\beta_i(\lambda), \beta_i,
\widehat\beta_{p+i}(\lambda)) - \frac1{p}\sum_{i=1}^p
f(\widehat\beta_i(\lambda_l), \beta_i, \widehat\beta_{p+i}(\lambda_l))
\right| = O(1/\sqrt{m})
\end{equation}
with probability tending to one as $p \goto \infty$. 

Now, for any $\lambda \in [\lambda_{\min}, \lambda_{\max}]$, choose
$l$ such that $\lambda_l \le \lambda < \lambda_{l+1}$ (set
$\lambda_{m+1} = \lambda_{\max} + \Delta$). Then from
\eqref{eq:m_union}, \eqref{eq:rho_contin}, and \eqref{eq:gap_small} we obtain
\begin{multline*}
\left| \frac1{p}\sum_{i=1}^p f(\widehat\beta_i(\lambda), \beta,
  \widehat\beta_{p+i}(\lambda)) - \ef(\lambda) \right|\\ \le \left|  \frac1{p}\sum_{i=1}^p f(\widehat\beta_i(\lambda), \beta,
  \widehat\beta_{p+i}(\lambda)) - \frac1{p}\sum_{i=1}^p
  f(\widehat\beta_i(\lambda_l), \beta, \widehat\beta_{p+i}(\lambda_l))
\right| \\
+ \left| \frac1{p}\sum_{i=1}^p f(\widehat\beta_i(\lambda_l), \beta,
  \widehat\beta_{p+i}(\lambda_l)) - \ef(\lambda_l) \right| +
|\ef(\lambda_l) - \ef(\lambda)|\\
\le O(1/\sqrt{m}) + \left| \frac1{p}\sum_{i=1}^p f(\widehat\beta_i(\lambda_l), \beta,
  \widehat\beta_{p+i}(\lambda_l)) - \ef(\lambda_l) \right| + \omega
\end{multline*}
holds uniformly for all $\lambda \in [\lambda_{\min}, \lambda_{\max}]$
with probability tending to one. Taking $m \goto \infty$, which allows us to set $\omega \goto 0$, gives
\[
\sup_{\lambda_{\min} \le \lambda \le \lambda_{\max}} \left|
  \frac1{p}\sum_{i=1}^p f(\widehat\beta_i(\lambda), \beta,
  \widehat\beta_{p+i}(\lambda)) - \ef(\lambda) \right| \Longrightarrow 0
\]
as $p \goto \infty$.

\end{proof}

\subsection{Proof of Theorem~\ref{thm:break}}
\label{sec:proof-crefthm:break}

The proof of Theorem~\ref{thm:break} presented here applies more generally to the Model-X knockoffs procedure that uses, instead of the LCD statistic, any other statistic of the form $W_j(\lambda) = w(\widehat{\beta}_j(\lambda), \widehat{\beta}_{j+p}(\lambda))$, where the link function $w$ satisfies $w(u, v) = -w(v, u)$ and $w(x, c) \goto \infty$ as $|x| \goto \infty$ for any fixed $c$; we call such $w$ function {\it faithful} in what follows. 
From \cite{su2017false} we know that Lasso cannot obtain full power unless \eqref{eq:transition} holds, hence we consider only the case $\epsilon < 2\epsilon^*(\delta/2)$. 
For any such $\epsilon$, it can be shown that the expressions in Equations \eqref{eq:tinfty} and \eqref{eq:overestimate} converge to $(1 - \epsilon)/(1 + \epsilon)$ when $t\to 0$ and $\Pi_m$ is growing as in the assumption. 
We consider first the case $q<(1 - \epsilon)/(1 + \epsilon)$. 

\smallskip
Let $\that > 0$ be the unique value of $t$ satisfying
\[
\widehat{\fdpinfty}^{\textnormal{LCD}}(t) = \frac{\P(w(\eta_{\alpha'\tau'}(\Pi + \tau' Z), \tau'\eta_{\alpha'}(Z')) \le -t)}{\P(w(\eta_{\alpha'\tau'}(\Pi + \tau' Z), \tau'\eta_{\alpha'}(Z')) \ge t)} = q. 
\]
When the prior distribution is $\Pi_m$, denote by $\alpha'_m, \tau'_m$ the solution to \eqref{eq:system-kf} and let $\that_m$ be defined as above. 
Recognizing the assumption of a growing $\Pi_m$ in Definition~\ref{def:esp-growing}, one can show that $\alpha'_m, \tau'_m$ converge to $\alpha'_{\infty}, \tau'_{\infty}$ which are the solution to
\begin{equation}\nonumber
\begin{aligned}
&\tau^2 = \sigma^2 + \frac{\epsilon \tau^2 (1+\alpha^2)}{\delta} + \frac{2-\epsilon}{\delta} \E \eta_{\alpha\tau}(\tau Z)^2\\
&\lambda = \left[1 - \frac{\epsilon}{\delta}- \frac{2-\epsilon}{\delta}\P(|\tau Z| > \alpha\tau)\right] \alpha\tau.
\end{aligned}
\end{equation}
That is, $\alpha'_m \goto \alpha'_{\infty}$ and $\tau'_m \goto \tau'_{\infty}$ as $m \goto \infty$. As a consequence, $\that_m$ tends to $\that_{\infty}$ as $m \goto \infty$ as well, where the existence of $\that_{\infty}$ is ensured by the fact that $0 < q < \frac{1 - \epsilon}{1 + \epsilon}$. 

Following the proof of Lemma A.1 in \cite{su2017false}, we can show that $\tpp(\lambda, \Pi_m, q)$ converges to
\[
\tppinfty(\lambda, \Pi_m, q) \equiv \P(w(\eta_{\alpha'_m \tau'_m}(\Pi_m + \tau'_m Z), \tau'_m \eta_{\alpha_m'}(Z')) \geq \that_m | \Pi_m \neq 0)
\]
in probability uniformly over $\lambda_1 \le \lambda \le \lambda_2$ as $n \goto \infty$, by making use of Theorem~\ref{thm:contrast_gen}. Having demonstrated earlier that $\alpha_m'$ and $\tau'_m$ converge to constants, the faithfulness of $w$ and the growing condition of $\Pi_m$ reveal that
\[
\tppinfty(\lambda, \Pi_m, q)  \goto 1.
\]
Moreover, the convergence of the probability $\tppinfty(\lambda, \Pi_m, q)$ as a smooth function of $\lambda$ to its limit 1 is uniform over $\lambda_1 \le \lambda \le \lambda_2$ as $m \goto \infty$. In particular, we can choose $m'$ such that
\begin{equation}\label{eq:last_hope1}
\inf_{\lambda_1 \le \lambda \le \lambda_2} \tppinfty(\lambda, \Pi_m, q) > 1 - \frac{\nu}{2}
\end{equation}
for all $m \ge m'$. Furthermore, for any $m$ we can find $n'(m)$ such that 
\begin{equation}\label{eq:last_hope2}
\sup_{\lambda_1 \le \lambda \le \lambda_2} \left|\tpp(\lambda, \Pi_m, q) - \tppinfty(\lambda, \Pi_m, q) \right| < \frac{\nu}{2}
\end{equation}
happens with probability at least $1 - \nu$ when $n \ge n'(m)$. Taken together, \eqref{eq:last_hope1} and \eqref{eq:last_hope2} ensure that, with probability at least $1 - \nu$, we have
\[
\inf_{\lambda_1 \le \lambda \le \lambda_2} \tpp(\lambda, \Pi_m, q) > 1 - \nu
\]
for $n \ge n'(m)$ and $m \ge m'$. 
When $q>(1 - \epsilon)/(1 + \epsilon)$, then for any $t>0$, the procedure that selects whenever $|W_j(\lambda)|>t$ asymptotically attains full power with $\fdpinfty(t)<q$,  and the assertion in the theorem holds. 
This concludes the proof.

As an aside, the proof above seamlessly carries over to any bridge-estimator-based knockoffs procedure that uses $W_j(\lambda) = |\widehat{\beta}_{j, \gamma}(\lambda)| - |\widehat{\beta}_{p+j, \gamma}(\lambda)|$~\cite{weng2018overcoming}. 
When the order $\gamma > 1$, in particular, the nominal level $q$ can take any value in $(0, 1)$ since the Donoho--Tanner phase transition does not occur once $\gamma > 1$.


\section{Derivation of the CV-AMP equations} \label{appdx:cv-amp}
Denote the minimum value for $\tau$ by 
\begin{equation*}
\taucv \equiv \min_{\lambda} \tau (\lambda; (K-1)\delta/K),
\end{equation*}
and let $\alphacv$ be the corresponding value for $\alpha$ (so $\alphacv$ is the solution in $\alpha$ to the first equation in \eqref{eq:system-kf} when $\tau$ replaced by $\taucv$). 
Note that we can characterize $(\alphacv, \taucv)$ by requiring that for $0 < t < \taucv$, 
\[
t^2 = \sigma^2 + \frac{K}{(K-1)\delta} \E\left[\eta_{\alpha t}(\Pi + t Z) - \Pi \right]^2 + \frac{K}{(K-1)\delta} \E \eta_{\alpha t}(t Z)^2
\]
does not have a solution in $t$ for $\alpha > \alpha_{\min}$. 
Therefore, on defining 
\[
f(u) \equiv \sigma^2 + \frac{K}{(K-1)\delta} \E\left[\eta_{ u \taucv}(\Pi + \taucv Z) - \Pi \right]^2 + \frac{K}{(K-1)\delta} \E \eta_{ u \taucv}(\taucv Z)^2 - \taucv^2,
\]
we are looking to solve
\begin{equation}\label{eq:alpha_uniq}
\frac{\d f(u)}{\d u} \Big|_{u = \alphacv} = 0.
\end{equation}
It is easy to verify, on the other hand, that
\[
\frac{\d f(u)}{\d u} = \frac{2\taucv^2 K}{(K-1)\delta} \left(\E\left[Z + u; \Pi + \tau Z < -\tau u \right] - \E\left[Z - u; \Pi + \tau Z > \tau u \right] \right) - \frac{4\taucv^2 K}{(K-1)\delta} \left[\phi(u) - u\Phi(-u)\right].
\]
Imposing now \eqref{eq:alpha_uniq}, we get the equation system
\[
\begin{aligned}
&\taucv^2 = \sigma^2 + \frac{K}{(K-1)\delta} \E\left[\eta_{\alphacv\taucv}(\Pi + \taucv Z) - \Pi \right]^2 + \frac{K}{(K-1)\delta} \E \eta_{\alphacv\taucv}(\taucv Z)^2\\
& 
\begin{split}
\frac{2\taucv^2 K}{(K-1)\delta} \left( \E\left[Z +\alphacv; \Pi + \taucv Z < -\taucv\alphacv \right] - \E\left[Z -\alphacv\right. \right. &\left. \left.; \Pi + \taucv Z > \taucv\alphacv \right]\right)\\
&\quad - \frac{4\taucv^2 K}{(K-1)\delta} \left[\phi(\alphacv) - \alphacv\Phi(-\alphacv) \right]=0,
\end{split}
\end{aligned}
\]
which simplifies to \eqref{eq:amp_cv}.

\end{document}